%% file: MPCLP-arXiv.tex
\documentclass[12pt]{article}
\usepackage[hmargin={1.75truecm,2truecm},vmargin={2truecm,2truecm}]{geometry}
\usepackage[english]{babel}
\usepackage[utf8]{inputenc}
\usepackage{geometry}
\usepackage{multicol,float}
\usepackage{multirow}
\usepackage[breaklinks,colorlinks,citecolor=blue,linkcolor=blue,hyperfootnotes=false]{hyperref}
\usepackage{amsmath}
\usepackage{amsthm}
\usepackage{amsfonts}
\usepackage{amssymb}
\usepackage{graphicx}
\usepackage{xcolor}
\usepackage{arydshln}
\usepackage[affil-it]{authblk}
\usepackage{natbib}

\usepackage{xspace}
\usepackage{booktabs}
\usepackage{subcaption}
\usepackage{orcidlink}
\usepackage{afterpage}

\usepackage[textsize=small,textwidth=3.3cm]{todonotes}
\usepackage[linesnumbered,ruled,vlined]{algorithm2e}

\usepackage{tikz}
\usetikzlibrary{datavisualization.formats.functions,shadows}
\usetikzlibrary{decorations.pathreplacing}
\usetikzlibrary{shapes.geometric}
\usetikzlibrary{arrows}
\usetikzlibrary{petri,calc}

\usepackage{cleveref}
\newtheorem{theorem}{Theorem}[section]
\newtheorem{example}[theorem]{Example}

\newtheorem{remark}[theorem]{Remark}
\newtheorem{proposition}[theorem]{Proposition}
\newtheorem{lemma}[theorem]{Lemma}

\Crefformat{appendix}{\S#2#1#3}
\crefname{theorem}{Theorem}{Theorems}
\crefname{example}{Example}{Examples}
\crefname{observation}{Observation}{Observations}
\crefname{remark}{Remark}{Remarks}
\crefname{proposition}{Proposition}{Propositions}
\crefname{lemma}{Lemma}{Lemmas}
\crefname{corollary}{Corollary}{Corollaries}
\crefname{algocf}{Algorithm}{Algorithms}	
\crefname{table}{Table}{Tables}	
\crefname{figure}{Figure}{Figures}
\crefname{algorithm}{Algorithm}{Algorithms}
\crefname{section}{Section}{Sections}
\crefname{algorithm}{Algorithm}{Algorithms}

\makeatletter
\def\blfootnote{\xdef\@thefnmark{}\@footnotetext}
\makeatother

\title{An efficient branch-and-cut algorithm for the multiple probabilistic covering location problem\blfootnote{
		\textit{Email addresses:} 
        \texttt{wangyanru@bit.edu.cn} (Yan-Ru Wang),
		\texttt{chenweikun@bit.edu.cn} (Wei-Kun Chen), 
		\texttt{ivana.ljubic@essec.edu} (Ivana Ljubi\'{c})}
}

\author[a]{Yan-Ru Wang\,\orcidlink{0009-0009-6256-2328}}
\author[a]{Wei-Kun Chen\,\orcidlink{0000-0003-4147-1346}}
\author[b]{Ivana Ljubi\'{c}\,\orcidlink{0000-0003-4147-1346}}

\affil[a]{\small School of Mathematics and Statistics, Beijing Institute of Technology, 100081, Beijing, China}
\affil[b]{\small Department of Information Systems, Data Analytics and Operations, ESSEC Business School, 95021 Cergy-Pontoise, France}

\date{}
\newif\ifarxiv
\arxivtrue 
\input{defpacks}

\newcommand{\keywords}[1]{\textbf{Keywords}: #1}

\begin{document}
	\maketitle
	
	\begin{abstract}
		In this paper, we consider the multiple probabilistic covering location problem (\MPCLP), which attempts to open a fixed number of facilities to maximize the total covered customer demand under a joint probabilistic coverage setting.
		We present a new mixed integer nonlinear programming (\MINLP) formulation, and develop an efficient linear programming (\LP) based branch-and-cut (\BnC) algorithm
		where submodular and outer-approximation inequalities are used to \rev{replace} the nonlinear constraints and {are} separated at the nodes of the search tree. 
		One key advantage of the proposed \BnC algorithm is that 
		the number of variables in the underlying formulation \rev{grows} only \rev{linearly} with the number of customers and facility locations and is one-order of magnitude smaller than that in the underlying formulation of a state-of-the-art \BnC algorithm in the literature.
		Moreover, we propose two \rev{new} families of strong valid inequalities, called \rev{\emph{enhanced outer-approximation}} and \rev{\emph{lifted subadditive inequalities}}, to strengthen the \LP relaxation and speed up the convergence of the proposed \BnC algorithm. 
		\rev{In} extensive computational experiments on a testbed of $240$ benchmark \MPCLP instances, we show that, thanks to the small problem size and the strong \LP relaxation of the underlying formulation, 
		the proposed \BnC algorithm significantly outperforms a state-of-the-art \BnC algorithm in terms of running time, number of nodes in the search tree, and number of solved instances. 
		In particular, using the proposed \BnC algorithm, we are able to provide optimal solutions for $57$ previously unsolved benchmark instances within a time limit of one hour.
		
	\end{abstract}
	

\input{introduction}

\input{formulations_v2}

	\input{linearization}

	\input{strong_valid_ineq_v2}
	\input{numerical_results}
	\input{conclusion}

	\bibliography{shorttitles,mpclp_abbr_noUIA}
	
	\newpage
	\appendix
	\input{sepa_submodular}
	\input{EOA_example}
	\input{auxilary_subadditive}
	\input{LSI_facet_defining}

\input{sepa_nphardness}

	\input{proof_heur_sepa}
	\input{alg_heur_sepa}
	\input{newly_solved_instances}

	

\end{document}

%% file: defpacks.tex
\usepackage{pdflscape}

\usepackage{etoolbox}

\newcommand{\be}{\text{$\boldsymbol{e}$}\xspace}
\DeclareMathOperator{\argmax}{\text{argmax}\xspace}
\DeclareMathOperator{\argmin}{\text{argmin}\xspace}
\usepackage{comment}
\makeatletter
\AddToHook{cmd/appendix/before}{\def\cref@section@alias{appendix}\def\cref@subsection@alias{appendix}}
\makeatother

\newif\iflastwasMath
\lastwasMathfalse

\newif\ifarxiv

\newcommand{\myqedhere}{\qedhere}
\newcommand{\mybeginproof}{\begin{proof}}
\newcommand{\myendqed}{}

\newenvironment{myproof}{%
	\lastwasMathfalse
	\mybeginproof
}{%
	\iflastwasMath
	\myqedhere
	\else
	\myendqed
	\fi
\end{proof}
}

\newcommand{\solver}[1]{\textsc{#1}\xspace}
\newcommand{\gurobi}{\solver{Gurobi}}
\newcommand{\cplex}{\solver{Cplex}}
\newcommand{\xpress}{\solver{Xpress}}
\newcommand{\scip}{\solver{Scip}}

\newcommand{\MINLP}{{MINLP}\xspace}

\newcommand{\MILP}{{MILP}\xspace}

\newcommand{\MIP}{{MIP}\xspace}

\newcommand{\LP}{{LP}\xspace}

\newcommand{\BnC}{B\&C\xspace}

\newcommand{\MPCLP}{\text{MPCLP}\xspace}
\newcommand{\MCLP}{\text{MCLP}\xspace}

\newcommand{\I}{\mathcal{I}}

\newcommand{\CL}{\mathcal{L}}
\newcommand{\J}{\mathcal{J}}
\newcommand{\W}{\mathcal{W}}
\newcommand{\X}{\mathcal{X}}
\newcommand{\Y}{\mathcal{Y}}
\newcommand{\CS}{\mathcal{S}}

\newcommand{\CO}{\mathcal{O}}

\newcommand{\C}{\mathcal{C}}

\newcommand{\Z}{\mathbb{Z}}
\newcommand{\R}{\mathbb{R}}

\newcommand{\id}{\texttt{id}}
\newcommand{\idfull}{\texttt{id-$r$-$R$-$\theta$}\xspace}
\newcommand{\V}{\text{$|\I|$}}
\newcommand{\K}{\text{$K$}}

\newcommand{\tbltime}{\texttt{T}\xspace}

\newcommand{\tblnode}{\texttt{N}\xspace}
\newcommand{\tblobj}{\texttt{OPT}\xspace}

\newcommand{\tblgap}{\texttt{G}(\texttt{\%})\xspace}
\newcommand{\tblrgap}{\texttt{LPG}(\texttt{\%})\xspace}

\newcommand{\tblnvar}{\texttt{V}\xspace}

\newcommand{\tblndata}{\texttt{\#}\xspace}

\newcommand{\tblnsol}{\texttt{S}\xspace}

\newcommand{\testBin}{\texttt{B\&C-B}\xspace}
\newcommand{\testInt}{\texttt{B\&C-I}\xspace}
\newcommand{\testBasicInt}{\texttt{vB\&C-I}\xspace}
\newcommand{\testBasicIntStr}{\texttt{vB\&C-I+E}\xspace}
\newcommand{\testBasicIntStrVI}{\texttt{vB\&C-I+E+L}\xspace}
\newcommand{\testBasicIntVI}{\texttt{vB\&C-I+L}\xspace}
\newcommand{\testLiterature}{\texttt{$\acute{\texttt{A}}$S19}\xspace}
\newcommand{\TL}{\texttt{TL}\xspace}

\DeclareMathOperator{\conv}{conv}

\newcommand{\Partial}{\text{P}}
\newcommand{\Full}{\text{F}}
\newcommand{\Phipartial}{\Phi_{\text{P}}}

\newcommand{\bzero}{\text{$\boldsymbol{0}$}\xspace}

\newcommand{\floor}[1]{\lfloor #1 \rfloor}

\newcommand{\rev}[1]{{#1}}

\newcommand{\U}{\mathcal{U}\xspace}

\usepackage{natbib}
\bibpunct[, ]{(}{)}{,}{a}{}{,}%
\def\BIBand{and}%

\makeatletter
\newcommand{\smartcitet}[1]{%
	\@tempcnta=0\relax 
	\@for\@citea:=#1\do{\advance\@tempcnta by 1\relax}
	\@tempcntb=0\relax 
	\@for\@citea:=#1\do{%
		\advance\@tempcntb by 1\relax 
		\citet{\@citea}
		\ifnum\@tempcntb<\@tempcnta 
		\ifnum\@tempcntb=\numexpr\@tempcnta-1\relax 
		\ifnum\@tempcnta>2 
		,\ \BIBand\ 
		\else
		\ \BIBand\ 
		\fi
		\else
		,\ 
		\fi
		\fi
	}%
}
\makeatother

%% file: introduction.tex
\section{Introduction}

The maximal covering location problem (\MCLP), proposed by \smartcitet{Church1974}, is a fundamental problem in the field of facility location.
The \MCLP attempts to locate $K$ facilities to serve (or cover) a set of customers while maximizing the total demand of covered customers.
This problem arises in or serves as a building block in a
wide variety of applications including
emergency services \citep{Wang2016}, 
healthcare systems \citep{Mendoza-Gomez2024},
ecological monitoring and conservation 
\citep{Martin-Fores2021},
telecommunications \rev{\citep{Cordeau2019,Nematzadeh2023}},  
\rev{influence maximization \citep{Coniglio2022}, 
electric vehicle charging station placement \citep{Lamontagne2023},
and 
competitive facility location \citep{Legault2025}.}
For a comprehensive discussion on \MCLP and its variants, we refer to the surveys \smartcitet{Snyder2011,Farahani2012,Garcia2019,Marianov2024} and the references therein.

A fundamental assumption in the \MCLP is that the coverage behavior is deterministic (also known as binary coverage); that is,
a customer is either fully covered by a facility or not covered at all, 
typically depending on whether the customer is within the coverage distance (or coverage radius) of the facility. 
However, this assumption may not be realistic in many real-life applications where the coverage behavior is usually probabilistic; see \smartcitet{Church1983,Berman2003,Budge2010,Daskin2013,Azizi2022} among many of them.
As an example, in emergency services, the ``coverage'' often refers to whether an emergency vehicle can reach a patient within a prespecified time threshold.
However, empirical findings from \smartcitet{Budge2010} indicate that such events occur with varying probabilities that are typically a non-increasing function of the distance between the emergency vehicle and the patient.
This introduces the notation of \emph{probabilistic coverage}.
One typical probabilistic coverage model is the gradual coverage model \citep{Church1983,Berman2003} where a customer is assumed to be fully covered by a facility within a predefined distance $r$, and not covered beyond distance $R$ (with $R > r$); and
between these two distances, however, the customer is partially covered, with the amount of coverage represented by a non-increasing function of the distance \citep{Berman2003,Berman2010,Drezner2014}. 
As noted in \smartcitet{Drezner2014,Berman2019}, it is often convenient to interpret the amount of coverage as the probability that full coverage will occur. 

When multiple facilities can jointly cover a customer, the following question arises immediately: how to combine the coverage from multiple facilities to the customer and estimate the aggregated coverage (also called cooperative coverage)?
To address this, various aggregation models have been proposed in the literature.
Following the traditional assumption in the \MCLP where each customer is served by the closest facility, 
\smartcitet{Berman2003} defined the aggregated coverage of a customer as the maximum individual probability among the open facilities.
This relies on the assumption that the coverage events from different facilities to a customer are totally correlated.
In contrast, \smartcitet{Berman2009,OHanley2013} assumed that the coverage events from different facilities to a customer are fully independent, and modeled the aggregated coverage as
one minus the probability that none of the open facilities can cover the customer (i.e., the probability that the customer is covered by at least one open facility).
To bridge these two extreme cases, \smartcitet{Drezner2014} introduced a \emph{joint coverage model}, which is a convex combination of the maximum probability of being covered by a single facility and
the overall probability of being covered by at least one facility. 
This joint coverage model allows for a flexible modeling framework of coverage dependency, and has been widely used in the literature \citep{Berman2019,Alvarez-Miranda2019,Michopoulou2024,Baldassarre2024}.

Due to the cooperative coverage, 
allowing \emph{co-location} (i.e., opening multiple facilities at the same location) 
could yield better overall coverage for customers,
in contrast to the classic \MCLP where co-location offers no benefits.
As such, co-location has been considered in various cooperative covering location problems \rev{\citep{Berman2007,Berman2011b,Berman2019,OHanley2013,Drezner2014,Michopoulou2024}}.

\smartcitet{Berman2019} proposed the multiple probabilistic covering location problem (\MPCLP) {in a gradual coverage context}
that integrates the probabilistic and cooperative coverage and co-location of the facilities into the framework of the \MCLP.
The \MPCLP attempts to open $K$ facilities (allowing for co-location) while maximizing the overall demand covered by the open facilities under the joint coverage model of \smartcitet{Drezner2014}.
\smartcitet{Alvarez-Miranda2019} developed an exact branch-and-cut (\BnC) approach for solving the \MPCLP, based on four mixed integer linear programming (\MILP) formulations.
However, in order to model the inherent co-location of facilities, these \MILP formulations require a huge number of binary variables ($K$ for each potential facility location),
making the \BnC algorithm inefficient to solve \MPCLP{s}, especially for problems with large $K$.

\subsection{Contributions and outline}

The goal of this paper is to propose a more compact mathematical formulation for the \MPCLP (in terms of a small number of variables) and develop
a linear programming (\LP) based \BnC algorithm that achieves high computational efficiency.
The main contributions of the paper are summarized as follows. 
\begin{itemize}
	\item 
	By introducing integer variables to represent the number of facilities co-located at the same location, we first present a mixed integer nonlinear programming (\MINLP) formulation whose number of variables \rev{grows} only \rev{linearly} with the number of customers and potential facility locations but does not grow with $K$. To tackle the proposed \MINLP formulation which is highly nonlinear,
    we derive an equivalent \MILP reformulation based on two families of valid linear inequalities (i.e., submodular inequalities \citep{Nemhauser1981} and outer-approximation inequalities \citep{Duran1986,Fletcher1994}), and develop an \rev{\LP-based} \BnC algorithm in which the derived valid inequalities are separated at the nodes of the search tree.
	\item Through a polyhedral investigation of a mixed integer set, a \rev{specific} substructure arising in the \MPCLP, we further develop two families of strong valid inequalities for the \MPCLP.
	The first ones, called \emph{enhanced outer-approximation inequalities}, are guaranteed to be stronger than outer-approximation inequalities.
	The second ones, called \emph{lifted subadditive inequalities}, 
	are shown to be facet-defining for the convex hull of the considered mixed integer set under a mild condition.
	Both of these proposed inequalities can effectively strengthen the \LP relaxation of the underlying formulation, 
	thereby \rev{enabling} the proposed \BnC algorithm to converge more quickly.	
\end{itemize}
Computational results on a testbed of 240 benchmark \MPCLP instances demonstrate that, thanks to the small problem size and the strong \LP relaxation of the underlying formulation, 
the proposed \BnC algorithm significantly outperforms the state-of-the-art \BnC algorithm of \smartcitet{Alvarez-Miranda2019}
in terms of running time, \rev{the} number of nodes in the search tree, and the number of solved instances. 
In particular, using the proposed \BnC algorithm, we are able to provide optimal solutions for $57$ previously unsolved benchmark instances within a time limit of one hour.

This paper is organized as follows.
\cref{sect:problem-description} provides an \MINLP formulation for the \MPCLP.
\cref{sect:reformulation} 
proposes an \rev{\LP-based} \BnC algorithm for solving the \MPCLP based on submodular and outer-approximation inequalities.
\cref{sect:valid-ineq} derives two families of strong valid inequalities to speed up the convergence of the proposed \BnC algorithm.
\cref{sect:numerical-experiments} reports the computational results.
Finally, \cref{sect:conclusion} draws the conclusion.

\subsection{Solution approaches}\label{subsec:literature-review}

In this subsection, we review the relevant references on the solution approaches for the \MPCLP and its two special cases where the coverage events from different facilities to a customer are assumed to be totally correlated and fully independent, respectively.

We first review the {\MPCLP} with totally correlated coverage behavior of the facilities where the aggregated coverage of a customer is equal to the maximum probability among the open facilities.
\smartcitet{Church1983} investigated the case where a stepwise function is used to estimate the coverage of a facility to a customer, and used an off-the-shelf \MILP solver to solve the formulated problem.
Subsequently, \smartcitet{Berman2002} considered the same case and developed an alternative \MILP formulation that has a smaller number of variables and constraints.
\smartcitet{Berman2003} and \smartcitet{Karasakal2004} investigated the case where a linear decay function is used to estimate the coverage of a facility to a customer, and solved the problem by off-the-shelf \MILP solvers and a Lagrangian relaxation based heuristic algorithm, respectively. 
It is worth remarking that the \MPCLP with totally correlated coverage behavior of the facilities is equivalent to the well-known $K$-median problem \citep{Berman2010}, which allows state-of-the-art \MILP\rev{-}based approaches in \rev{\smartcitet{Avella2007,Garcia2010,Duran-Mateluna2023,Ljubic2024,Cherkesly2025}} to find a \rev{globally} optimal solution.

Next, we review the \MPCLP with fully independent coverage behavior of the facilities where the aggregated coverage of a customer is equal to one minus the probability that none of the open facilities can cover the customer.
\smartcitet{Daskin1983} investigated the case with uniform coverage probabilities and formulated the problem as a compact \MILP formulation.
Subsequently, \smartcitet{Aytug2002} and \cite{Rajagopalan2007} developed genetic and meta-heuristic algorithms (including evolutionary, tabu search, simulated annealing, and hill-climbing algorithms) to find a feasible solution, respectively.
For the \MPCLP with non-uniform coverage probabilities of the facilities, \smartcitet{Polasky2000} and \smartcitet{Camm2002} investigated the case arising in the nature reserve site selection and developed the greedy and \MILP approximation algorithms, respectively.
\smartcitet{OHanley2013} developed {the probability chain technique} to reformulate the problem as a compact \MILP formulation that can be solved to optimality by general-purpose \MILP solvers. 
However, due to the large problem size, only instances with up to $100$ potential facility locations and a $K$ value of up to $40$ were reported.

For the general case where the aggregated coverage is a convex combination of 
the maximum probability of being covered by a single facility and 
the overall probability of being covered by at least one facility \citep{Drezner2014}, 
\smartcitet{Berman2019} proposed several heuristic algorithms, including a tangent line approximation method, a greedy heuristic, and an ascent and tabu search algorithm, to find a feasible solution for the problem.
Using the submodularity of the objective function, \smartcitet{Alvarez-Miranda2019} developed a \BnC approach for solving the \MPCLP based on four \MILP formulations.
However, due to the huge problem size and the large search tree size, the \BnC approach is usually inefficient;
for a testbed of $240$ benchmark instances with numbers of customers and potential facility locations ranging from $100$ to $900$ and $K$ values ranging from $5$ to $200$, only $148$ instances can be solved to optimality by the \BnC approach of \smartcitet{Alvarez-Miranda2019} within a time limit of one hour.	

%% file: formulations_v2.tex
\section{Problem formulation}
\label{sect:problem-description}

In this section, we first briefly review the \MPCLP of \cite{Berman2019} and then present an \MINLP formulation for it.
\subsection{Problem definition}\label{subsec:problem-definition}

In the considered problem, we are \rev{given} a set of potential facility locations $\I$ and a set of customers $\J$, with a demand $d_j \rev{>0}$ for each $j \in \J$.
The probability of customer $j \in \J$ being covered by a facility at location $i \in \I$ is denoted as $p_{ij}$. 
Let $\CS \subseteq \I^K$ be a subset of open facilities, where $\I^K := \I \times [K]$ (with $[K]= \{1, \ldots,K\}$).
Here, $(i,k) \in \I^K$ denotes the $k$-th (open) facility at location $i$ (throughout, co-locating multiple facilities at a single location is allowed).
In a probabilistic setting, the cooperative coverage naturally takes the value of the joint probability of a customer being covered by at least one open facility \citep{Berman2019}.
If the coverage events are totally correlated, then the cooperative coverage is equal to the maximum individual probability among these facilities, i.e., $\max_{(i,k) \in \CS} p_{i j}$; 
conversely, if the coverage events are independent, then the cooperative coverage can be calculated as $1 - \prod_{(i,k) \in \CS} (1 - p_{ij})$.
To model the spectrum of dependency of the coverage events,
a convex combination of the above two extreme cases, called the \emph{joint coverage function}, is used to measure the cooperative coverage \citep{Drezner2014,Berman2019,Alvarez-Miranda2019,Michopoulou2024,Baldassarre2024}: 
\begin{equation}\label{JCF-set}
	p_j(\theta, \CS) = \theta \left(\max_{(i,k) \in \CS} p_{ij}\right) 
	+ (1 - \theta) \left(1 - \prod_{(i,k) \in \CS} (1-p_{ij})\right).
\end{equation}
Here, parameter $\theta \in [0,1]$ weighs between the two extreme cases of total correlation ($\theta = 1$) and full independence ($\theta = 0$), thereby allowing for flexible modeling of coverage dependency.

The \MPCLP
attempts to open $K$ facilities so that the total demand cooperatively covered by the open facilities is maximized and can be formally stated as 
\begin{align}\label{mgclp-Set-MINLP}\tag{MPCLP}
	\max_{\CS \subseteq \I^K, \, |\CS| \le K}  \sum_{j\in \J} d_j \left(\theta \left(\max_{(i,k) \in \CS} p_{ij}\right) 
	+ (1 - \theta) \left(1 - \prod_{(i,k) \in \CS} (1-p_{ij})\right) \right).
\end{align}
When $\theta = 1$, \eqref{mgclp-Set-MINLP} reduces to the case with totally correlated coverage behavior of the facilities \citep{Church1983,Berman2002,Berman2003,Karasakal2004}; 
when $\theta= 0$, it reduces to the case with fully independent coverage behavior of the facilities \citep{Polasky2000,Camm2002,OHanley2013}.

Note that as the objective function of \eqref{mgclp-Set-MINLP} is non-decreasing with respect to $\CS$, we can equivalently state the cardinality constraint $|\CS|=K$ as $|\CS| \le K$.
Also note that if $K=1$, then determining the optimal solution of problem  \eqref{mgclp-Set-MINLP} is trivial --- one can simply compute $\sum_{j \in \J} \rev{d_j} p_{j}(\theta, \{(i,1)\})$ for each $i \in \I$ and choose the facility corresponding to the best objective value. 
Therefore, for simplicity of discussion, we assume $K \geq2$ throughout the paper.

\subsection{An \MINLP formulation for \eqref{mgclp-Set-MINLP}}\label{subsect:MINLP-formulation}

Next, we present an \MINLP formulation for \eqref{mgclp-Set-MINLP}. 
To achieve this, we first 
notice that facilities $(i,k)$, $k \in [K]$, at the same location $i$ exhibit identical coverage probabilities $p_{ij}$.
Thus, 
the joint coverage function $p_j(\theta, \CS)$ in \eqref{JCF-set} can be rewritten as 
\begin{equation}\label{JCF-set-re}
	p_j(\theta, \CS) = \theta \left(\max_{i \in \I \,: \,\CS_i \neq \varnothing} p_{ij}\right) 
	+ (1 - \theta) \left(1 - \prod_{i \in \I} (1-p_{ij})^{|\CS_i|}\right),
\end{equation}
where $\CS_i := \{(\ell,k) \in \CS: \ell = i\}$ is the set of open facilities at location $i \in \I$.
This motivates us to introduce, for each $i \in \I$,
a binary variable $z_i$ indicating whether $\CS_i$ is nonempty or not, 
and an integer variable $y_i$ denoting the number of facilities $|\CS_i|$ at location $i$, and rewrite the set optimization formulation \eqref{mgclp-Set-MINLP} as an \MINLP formulation:
\begin{align}
	\label{mgclp-Int-MINLP}\tag{MINLP}
		\max_{y,\,z} \left\{ \sum_{j\in \J}  d_j \left (\theta \left(\max_{i \in \I} p_{ij} z_i\right) + (1 - \theta) \left(1 - \prod_{i\in \I} (1-p_{ij})^{y_i}\right)\right) \, : \, \eqref{cons:card-y}\text{--}\eqref{cons:z-domain}\right\},
\end{align}
where 
\begin{subequations}
	\begin{align}
		& \sum_{i\in \I } y_i \le K, \label{cons:card-y}\\
		& z_i \le y_i \le K z_i , \ \forall \ i \in \I , \label{cons:zi-yi}\\
		&  y_i \in \mathbb{Z}_{\rev{\ge 0}}, \ \forall \ i \in \I , \label{cons:y-domain}\\
		&  z_i \in \{0,1\}, \ \forall \ i \in \I. \label{cons:z-domain}
	\end{align}
\end{subequations}
Constraint \eqref{cons:card-y} restricts the total number of open facilities to be less than or equal to $K$.
Constraints \eqref{cons:zi-yi} together with the integrality constraints \eqref{cons:y-domain}--\eqref{cons:z-domain} show that $z_i=1$ if and only if $1 \leq y_i \leq K$.
Throughout the paper, we adopt the convention that $0^0=1$ (as to address the case $p_{ij}=1$ and $y_i=0$).

Two remarks on the proposed \eqref{mgclp-Int-MINLP} are in order.
First, \eqref{mgclp-Int-MINLP} 
can be solved using general-purpose \MINLP solvers such as \gurobi and \scip.
However, due to the highly nonlinear objective function ({particularly} the {exponential} part), 
our preliminary experiments show that directly solving \eqref{mgclp-Int-MINLP} with a general-purpose \MINLP solver is inefficient, especially for the large-scale cases. 
In the next section, we will develop an efficient \BnC algorithm based on an \MILP reformulation.
Second, to model the co-location of facilities, one can introduce, for each $i\in \I$ and $k \in [K]$, a binary variable $x_i^k$ to denote whether the $k$-th facility at location $i$ is opened or not \citep{OHanley2013,Alvarez-Miranda2019,Michopoulou2024}. 
However, this leads to an \MINLP formulation with at least $|\I|K$ binary variables. 
In sharp contrast, formulation \eqref{mgclp-Int-MINLP} only involves the $2|\I|$ variables $\{y_i\}_{\rev{i \in \I}}$ and $\{z_i\}_{\rev{i \in \I}}$, which is one order of magnitude smaller, thereby facilitating the design of efficient algorithms (such as the \BnC algorithm \rev{introduced below}) for solving the \MPCLP.

\section{A branch-and-cut  algorithm based on formulation \eqref{mgclp-Int-MINLP}}
\label{sect:reformulation}

To develop a \BnC approach for solving problem \eqref{mgclp-Int-MINLP}, let us rewrite 
problem \eqref{mgclp-Int-MINLP} in the hypograph forms of $ \max_{i \in \I} p_{ij} z_i $ and $ 1 - \prod_{i\in \I} (1-p_{ij})^{y_i}$:
\begin{align}\tag{MINLP'}	\label{mgclp-Int-MINLP2}
	\max_{y,\,z,\,\zeta,\,\eta} \left\{ \sum_{j\in \J} d_j \left(\theta \zeta_j + (1-\theta)\eta_j\right) \, : \,  \eqref{cons:card-y}\text{--}\eqref{cons:z-domain},~\eqref{cons:zeta-nonlinear1},~\eqref{cons:eta-nonlinear1}\right\}, 
\end{align}
where
\begin{subequations}
	\begin{align}
		&  \zeta_j  \leq \max_{i \in \I} p_{ij} z_i, \ \forall \ j \in \J, \label{cons:zeta-nonlinear1}\\
		& \eta_j  \leq 1 - \prod_{i\in \I} (1-p_{ij})^{y_i}, \ \forall \ j \in \J. \label{cons:eta-nonlinear1}
	\end{align}
\end{subequations}
Here, for each $j \in \J$, $\zeta_j$ and $\eta_j$ measure the nonlinear functions $ \max_{i \in \I} p_{ij} z_i$ and $1 - \prod_{i\in \I} (1-p_{ij})^{y_i}$, respectively. 
In the following, we will first derive linearization techniques to \rev{replace} the nonlinear constraints \eqref{cons:zeta-nonlinear1} and \eqref{cons:eta-nonlinear1} in problem \eqref{mgclp-Int-MINLP2} \rev{with} linear inequalities, and then  develop an \rev{\LP-based} \BnC algorithm where the derived linear inequalities are separated at the nodes of the search tree.
We also compare the proposed \BnC algorithm with the state-of-the-art \BnC algorithm of \cite{Alvarez-Miranda2019}.

Throughout, for simplicity of notations, we drop the subscript $j$ in \eqref{cons:zeta-nonlinear1} and \eqref{cons:eta-nonlinear1}, and consider the following constraints:
\begin{align}
	&  \zeta \leq \max_{i \in \I} p_i z_i, \label{cons:zeta-nonlinear} \\
	& \eta  \leq 1 - \prod_{i\in \I} (1-p_i)^{y_i}.\label{cons:eta-nonlinear}
\end{align}

\subsection{Linearization of constraint \eqref{cons:zeta-nonlinear}} \label{subsect:linearization-max}

Using the results in \citet[Theorem 8]{Nemhauser1981}, constraint \eqref{cons:zeta-nonlinear} in problem \eqref{mgclp-Int-MINLP2} can be equivalently \rev{replaced with} the following $|\I|+1$ linear inequalities (called \emph{submodular inequalities}): 
\begin{equation}\label{cons:zeta-linear}
	\zeta \le p_{\ell} + \sum_{i\in \I} (p_i - p_{\ell})^+ z_i, \ \forall \ \ell \in \{0\}\cup \I,
	\end{equation}
where $p_0=0$ and ${(\cdot)}^+= \max\{0,\cdot\}$; see also \smartcitet{Alvarez-Miranda2019}.
Inequalities \eqref{cons:zeta-linear} guarantee that 
$\zeta$ is upper-bounded by the maximum probability among the open facilities.
Specifically, for $z \in\{0,1\}^{|\I|}$, if $ z_i=0$ for all $i \in \I$, this upper bound is $p_0=0$; otherwise,  
this upper bound is $p_{\ell_0}$ where $\ell_0 \in \argmax_{i \in \I}p_i z_i$ (which is achieved by inequality \eqref{cons:zeta-linear} with $\ell = \ell_0$).
Thus, $\max_{i\in \I} p_i z_i = \min_{\ell \in \{0\}\cup \I} p_{\ell} + \sum_{i\in \I}(p_i - p_{\ell})^+ z_i$ holds for all $z \in \{0,1\}^{|\I|}$, 
and the equivalence of \rev{the representations of} \eqref{cons:zeta-nonlinear} and \eqref{cons:zeta-linear} in problem \eqref{mgclp-Int-MINLP2} follows.

A variant class of inequalities \eqref{cons:zeta-linear}, which is used to represent $\min_{i \in \I} p_i z_i$ in the objective of minimization problems, 
was developed by \smartcitet{Magnanti1981,Cornuejols1990,Fischetti2017} for the uncapacitated facility location problem. 
\smartcitet{Cornuejols1980,Duran-Mateluna2023} used these inequalities for solving the {$K$-median} problem\rev{, and \smartcitet{Ljubic2024} used similar inequalities for solving a class of discrete ordered median problems}.

\subsection{Linearization of constraint \eqref{cons:eta-nonlinear}}	\label{subsect:linearization-prod}

One well-known approach for linearizing a nonlinear constraint of the form $\eta \leq f(y)$ is to use the supporting hyperplanes of the hypograph of function $f(y)$ at all points $y$.
This approach, known as the outer-approximation method \citep{Duran1986,Fletcher1994}, relies on the property that $f(y)$ is a concave and continuously differentiable function.
Unfortunately, function $\Phi(y) := 1 - \prod_{i \in \I} (1-p_i)^{y_i}$, the right-hand \rev{side} of the nonlinear constraint \eqref{cons:eta-nonlinear} in problem \eqref{mgclp-Int-MINLP2}, is generally not differentiable, as demonstrated in \cref{ex1}.
\begin{example}
	\label{ex1}
	Let $|\I| = 1$ and $p = 1$.
	Then function $\Phi(y) = 1 - (1 - p)^{y} = 1 - 0^{y}$
	takes the value of $0$ if $y = 0$ and $1$ if $y > 0$.
	As a result, $\lim_{y \rightarrow 0^+} \Phi(y) = 1 \neq 0 = \Phi(0)$, which implies  
	that $\Phi(y)$ is not right-continuous at $y = 0$, and thus is not differentiable at $y = 0$.
\end{example}

The reason for the nondifferentiablity of function $\Phi(y)$ is that $p_i =1$ holds for some $i \in \I$.
Indeed, if $p_i < 1$ holds for all $i \in \I$, then function $\Phi(y)$ is {continuously differentiable} as it can be expressed as {the composition of two smooth functions: 
	$h(g) = 1 - e^{-g}$ and $g(y) = \sum_{i \in \I} [- \ln(1-p_i)] y_i$}.
However, the case $p_i=1$ for some $i \in \I$ (i.e., the customer is fully covered by some facility location) frequently arises in real applications. 
As an example, in the gradual coverage context,
a customer  is considered to be fully covered by a facility  if the distance between them is less than or equal to the prespecified minimum coverage distance $r$ \citep{Church1983,Berman2003}.
							
To bypass the difficulty posed by the nondifferentiablity of function $\Phi(y)$ in problem \eqref{mgclp-Int-MINLP2}, we introduce the following function  
\begin{equation}\label{def:tilde-Phi}
	\tilde{\Phi}(y) :=
	1 - \prod_{i\in \I_{\Partial}} (1 - p_i)^{y_i} 
	+ \sum_{i\in \I_{\Full}} y_i,
\end{equation}
where $\I_{\Partial} := \{i \in \I\,: \,0 \le p_i < 1\}$ and $\I_{\Full} := \{i \in \I\, :\, p_i = 1\}$ are the sets of facilities that can partially and fully cover the customer, respectively.
It is easy to see that $\tilde{\Phi}(y)$ is differentiable at all $y \in \mathbb{Z}_{\rev{\ge 0}}^{|\I|}$.
Moreover, \begin{theorem}\label{cor:linearized-prod-term}
	Let 
	\begin{align}
		& \U = \left\{ 
		(\eta, y) \in \R \times \rev{\mathbb{Z}^{|\I|}_{\geq 0}}: \eta \le \Phi(y) = 1 - \prod_{i\in \I}(1 - p_i)^{y_i}
		\right\}, \label{Sdef}\\
		& \tilde{\U} = \left\{
		(\eta, y) \in \R \times \rev{\mathbb{Z}^{|\I|}_{\geq 0}}: \eta \le \tilde{\Phi}(y)
		= 1 - \prod_{i\in \I_{\Partial}} (1 - p_i)^{y_i} 
		+ \sum_{i\in \I_{\Full}} y_i, \ 
		\eta \le 1
		\right\},\label{tSdef}
	\end{align}
	Then it follows that $\U = \tilde{\U}$.
	\end{theorem}
\begin{myproof}
It suffices to show that $\Phi(y) = \min\{1, \tilde{\Phi}(y)\}$ for all $y \in \rev{\mathbb{Z}^{|\I|}_{\geq 0}}$.
By $0 \le p_i < 1$ for $i \in \I_{\Partial}$, we have $0 < \prod_{i \in \I_{\Partial}} (1-p_i)^{y_i} \le 1$ for  $y \in \Z_{\rev{\ge 0}}^{|\I|}$.
Thus, 
\begin{itemize}
	\item [(i)] if $y_i = 0$ for all $i \in \I_{\Full}$, then $\Phi(y) = \tilde{\Phi}(y)= 1- \prod_{i \in \I_{\Partial}} (1-p_i)^{y_i}  < 1$; 
	\item [(ii)] if $y_i \geq 1$ for some $i \in \I_{\Full}$, then $\Phi(y) = 1 \leq 1- \prod_{i \in \I_{\Partial}} (1-p_i)^{y_i}  + 1 \leq 1- \prod_{i \in \I_{\Partial}} (1-p_i)^{y_i} + \sum_{i \in \I_{\Full}} y_i = \tilde{\Phi}(y)$.
\end{itemize}
In both cases, we have $\Phi(y) = \min\{1, \tilde{\Phi}(y)\}$.
\end{myproof}

\cref{cor:linearized-prod-term} indicates that the {representation} of $\eta \leq \Phi(y)$ and that of $\eta \leq \tilde{\Phi}(y)$ and $\eta \leq 1$ in problem \eqref{mgclp-Int-MINLP2} are \emph{equivalent}. 
This implies that to linearize $\eta \leq \Phi(y)$ with the nondifferentiable function $\Phi(y)$, we can equivalently linearize $\eta \leq \tilde{\Phi}(y)$ with the continuously {differentiable} function $\tilde{\Phi}(y)$.
Moreover, $h(g(y)) := 1 - \prod_{i\in \I_{\Partial}} (1 - p_i)^{y_i} $ is a concave function over $y$ as $h(g) = 1 - e^{-g}$ is concave and non-decreasing, and $g(y) = \sum_{i \in \I_{\Partial}} [- \ln(1-p_i)] y_i$ is linear \citep{Boyd2004}, 
implying that
$\tilde{\Phi}(y) = 	1 - \prod_{i\in \I_{\Partial}} (1 - p_i)^{y_i} 
+ \sum_{i\in \I_{\Full}} y_i$ is also a concave function.
The {differentiability} and concavity of $\tilde{\Phi}(y)$ enable to equivalently represent the nonlinear constraint $\eta \leq \tilde{\Phi}(y)$ as linear inequalities using the outer-approximation method \citep{Duran1986,Fletcher1994}, as detailed in the following.

Given a vector \rev{$y^* \in \rev{\mathbb{Z}^{|\I|}_{\geq 0}}$},  from the concavity and differentiability of $\tilde{\Phi}(y)$, 
we can bound the value of $\tilde{\Phi}(y)$ from above by its first-order approximation at $y^*$, obtaining the following valid inequality for $\rev{\tilde{\U}}$:
\begin{equation}\label{cons:prod-outer-approx}
	\eta 
	\le \tilde{\Phi}(y^*) 
	+ \sum_{i\in \I} \frac{\partial \tilde{\Phi}}{\partial y_i} (y^*) \cdot  (y_i - y_i^*).
\end{equation}
Note that
\begin{equation*}
	\frac{\partial \tilde{\Phi}}{\partial y_i} (y^*)
	= - [\ln(1-p_i)] \prod_{\ell \in \I_{\Partial}}(1-p_{\ell})^{y_{\ell}^*} , ~\forall~ i \in \I_{\Partial},~\text{and}
	~\frac{\partial \tilde{\Phi}}{\partial y_i} (y^*)
	=1,~\forall~i \in \I_{\Full}.
\end{equation*}
Thus, the right-hand side of inequality \eqref{cons:prod-outer-approx} reduces to
\begin{equation*}
	\footnotesize
	\begin{aligned}
		& \tilde{\Phi}(y^*) 
		+ \sum_{i\in \I} \frac{\partial \tilde{\Phi}}{\partial y_i} (y^*) (y_i - y_i^*)\\
		& = 1 - \prod_{i\in \I_{\Partial}}(1-p_i)^{y_i^*} + \sum_{i\in \I_{\Full}} y_i^*
		- \sum_{i\in \I_{\Partial}}  [\ln(1-p_i)] \cdot \prod_{\ell \in \I_{\Partial}}(1-p_{\ell})^{y_{\ell}^*}  \cdot 
		(y_i -  y_i^*)
		+ \sum_{i\in \I_{\Full}}(y_i -  y_i^*)\\
		& = 1 - \prod_{i\in \I_{\Partial}}(1-p_i)^{y_i^*} + \sum_{i\in \I_{\Partial}}\left([\ln(1-p_i)] \prod_{\ell \in \I_{\Partial}}(1-p_{\ell})^{y^*_{\ell}}\right)  y_i^* -\sum_{i\in \I_{\Partial}}\left([\ln(1-p_i)] \prod_{\ell \in \I_{\Partial}}(1-p_{\ell})^{y^*_{\ell}}\right)y_i  + \sum_{i\in \I_{\Full}} y_i.
	\end{aligned}
\end{equation*}
Therefore, inequality \eqref{cons:prod-outer-approx} can be written as 
\begin{align}\label{cons:prod-outer-approx-explicit}
			\eta \le 
	c(y^*) + \sum_{i\in \I_{\Partial}} a_i(y^*) y_i + \sum_{i\in \I_{\Full}} y_i,
\end{align}
where 
\begin{align}
	& c(y^*)  :=  1 - \prod_{i\in \I_{\Partial}}(1-p_i)^{y_i^*} +   \left(\prod_{\ell \in \I_{\Partial}}(1-p_{\ell})^{y^*_{\ell}} \right)\sum_{i\in \I_{\Partial}}[\ln(1-p_i)] y_i^*,
	\label{def:c}\\
	& a_i(y^*)  := - [\ln(1-p_i)] \prod_{\ell \in \I_{\Partial}}(1-p_{\ell})^{y^*_{\ell}}, 
	\  \forall \ i \in \I_{\Partial}.
	\label{def:ai}
\end{align}
\begin{remark}\label{remark1}
	Let $\Y = \{y \in  \rev{\mathbb{Z}^{|\I|}_{\geq 0}}\, : \, \sum_{i \in \I} y_i \leq K \}$. Then $\left\{ (\eta, y) \in \R \times {\Y}\, :\,\eta \le \Phi(y) = 1 - \prod_{i\in \I}(1 - p_i)^{y_i} \right\}=\left\{(\eta, y) \in \R \times \rev{\Y}\, :\, \eta \leq c(y^*) + \sum_{i\in \I_{\Partial}} a_i(y^*) y_i + \sum_{i\in \I_{\Full}} y_i,~\rev{\forall~{y^* \in \Y}},~\eta \leq 1\right\}$.
\end{remark}
\rev{From} \cref{cor:linearized-prod-term} and \cref{remark1}, \rev{it follows} that the nonlinear constraint \eqref{cons:eta-nonlinear} in problem \eqref{mgclp-Int-MINLP2} can be equivalently \rev{represented} as the linear outer-approximation inequalities \eqref{cons:prod-outer-approx-explicit} (for all $y^*\in \rev{\Y}$) and $\eta\leq 1$.

\subsection{A branch-and-cut approach based on submodular and outer-approximation inequalities}
\label{subsec:MILP}

From \cref{subsect:linearization-max,subsect:linearization-prod}, we can develop the following equivalent \MILP formulation for the \MPCLP based on submodular inequalities \eqref{cons:zeta-linear} and outer-approximation inequalities \eqref{cons:prod-outer-approx-explicit}:
\begin{align}
		\tag{MILP}\label{mgclp-Int-MILP}
	\max_{y,\,z,\,\zeta,\,\eta} \left\{\sum_{j \in \J} d_j \left(\theta \zeta_j + (1 - \theta) \eta_j \right)\, : \, \eqref{cons:card-y}\text{--}\eqref{cons:z-domain},~\eqref{cons:zeta-linear-all},~\eqref{cons:eta-linear},~\eqref{cons:eta-domain}\right\},
\end{align}
where
\begin{subequations}
	\begin{align}
		& \zeta_j \le p_{\ell j} + \sum_{i \in \I} (p_{ij}- p_{\ell j})^+ z_i, \ \forall 
		\  j \in \J,~\forall~ \ell \in \{0\} \cup \I, \label{cons:zeta-linear-all}\\
		& \eta_j \le c_{j}(y^*) + \sum_{i\in \I_{j, \Partial}} a_{ij}(y^*) y_i + \sum_{i\in \I_{j, \Full}} y_i, \  
		\ \forall \ j \in \J, 
		\ \forall \ y^* \in \rev{\Y}, 
		\label{cons:eta-linear}\\
		& \eta_j \le 1, \ \forall \ j \in \J. \label{cons:eta-domain}
	\end{align}
\end{subequations}
Here, $\Y=\{y \in  \rev{\mathbb{Z}^{|\I|}_{\geq 0}}\, : \, \sum_{i \in \I} y_i \leq K \}$, $p_{0j}=0$, $\I_{j, \Partial} =\{ i \in \I \, : \, 0 \le p_{ij}<1\}$, and $\I_{j, \Full} = \{ i \in \I \, : \, p_{ij}=1 \}$ for $j \in \J$; and
the submodular inequalities \eqref{cons:zeta-linear-all},  outer-approximation inequalities \eqref{cons:eta-linear}, and $\eta_j \le 1$ for $j \in \J$
are the linearizations of the nonlinear constraints \eqref{cons:zeta-nonlinear1} and \eqref{cons:eta-nonlinear1},  respectively, with subscript $j$.

Although \eqref{mgclp-Int-MILP} is an \MILP problem, its quadratic number of {constraints} \eqref{cons:zeta-linear-all}, and particularly the exponential number of {constraints} \eqref{cons:eta-linear} make it impractical to be directly solved by modern \MILP solvers such as \gurobi, \cplex, \rev{or} \xpress.  
To overcome this weakness, we implement a \BnC algorithm 
in which constraints \eqref{cons:zeta-linear-all} and \eqref{cons:eta-linear} are separated at the nodes in the search tree.
Given an \LP relaxation solution $(y^*, z^*, \zeta^*, \eta^*)$ encountered at a node in the search tree, we separate inequalities \eqref{cons:zeta-linear-all} and \eqref{cons:eta-linear} using the following procedures.

\begin{itemize}
	\item 
	\rev{For each fixed $j \in \J$,} determining whether there exists \rev{an inequality of \eqref{cons:zeta-linear-all} violated by $(y^*, z^*, \zeta^*, \eta^*)$} can be done by checking the corresponding $|\I| +1$ inequalities, and thus the computational complexity is $\CO(|\I|^2)$.
	However, using a similar idea as in \smartcitet{Fischetti2017,Duran-Mateluna2023}, we can find the most violated inequality with an improved complexity of $\CO(|\I|)$. 
	For completeness, we provide a detailed proof in 	Appendix \ref*{appendix:separation-zeta-linear}.
									
	\item \rev{For each fixed $j \in \J$}, 
    if {$y^* \in \mathbb{Z}_{\ge 0}^{|\I|}$}, 
    the exact separation of inequalities \eqref{cons:eta-linear} can be done by first computing the coefficients $a_{ij}(y^*)$ for $i \in \I_{j, \Partial}$ and the constant $c_{j}(y^*)$ in \eqref{def:c}--\eqref{def:ai} (where subscript $j$ is omitted), and then checking whether inequality $\eta_j \le c_{j}(y^*) + \sum_{i\in \I_{j, \Partial}} a_{ij}(y^*) y_i + \sum_{i\in \I_{j, \Full}} y_i$ is violated.
	Clearly, this can be performed in $\CO(|\I|)$ time.
	If $y^* \notin \mathbb{Z}_{\rev{\ge 0}}^{|\I|}$, 
    we use, however, a heuristic procedure to separate inequalities \eqref{cons:eta-linear}, also with a complexity of $\CO(|\I|)$. 
	Specifically, we first round $y^*$ to its nearest integer vector $\tilde{y}$ (i.e,  $\tilde{y}_i= \lfloor y_i^*\rfloor$ if $y_i^* - \lfloor y_i^*\rfloor < 0.5$ and $\tilde{y}_i= \lceil y_i^*\rceil$ otherwise) 
	and compute coefficients $a_{ij}(\tilde{y})$ for $i \in \I_{j, \Partial}$ and the constant $c_{j}(\tilde{y})$. 
	Then, we add the inequality  $\eta_j \le c_{j}(\tilde{y}) + \sum_{i\in \I_{j, \Partial}} a_{ij}(\tilde{y}) y_i + \sum_{i\in \I_{j, \Full}} y_i$  if it is violated by  $(y^*, z^*, \zeta^*, \eta^*)$.
\end{itemize}
Note that the exact separation of inequalities \eqref{cons:zeta-linear-all} and \eqref{cons:eta-linear} at points with integral vectors $(y^*, z^*)$ is enough to ensure the convergence and correctness of the \BnC algorithm for problem \eqref{mgclp-Int-MILP} (which follows from the fact that there are at most $(|\I|+1)|\J|$ and {$\sum_{k=0}^{K} {|\I|}^{k} |\J|$} inequalities of  \eqref{cons:zeta-linear-all} and \eqref{cons:eta-linear}, respectively).
Nevertheless,
we can speed up the convergence of the \BnC algorithm by constructing violated inequalities \eqref{cons:zeta-linear-all} and \eqref{cons:eta-linear} {at fractional solutions} encountered inside the search tree.
In our implementation, we separate both inequalities {at the nodes} in the \BnC search tree.

\subsection{Comparison with the state-of-the-art branch-and-cut approach  of \citet{Alvarez-Miranda2019}}\label{subsec:formulation-comparison}
\rev{By leveraging the submodularity} of the set functions  $\max_{(i,k) \in {\CS}} p_{ij}$ and $1 - \prod_{(i,k) \in \CS} (1-p_{ij})$ (where $j \in \J$ and $\CS\subseteq \I ^{K}$) in \eqref{mgclp-Set-MINLP} and the technique of binary representations of the integer variables \citep{Dash2018}, \smartcitet{Alvarez-Miranda2019} developed four \MILP formulations for the \MPCLP.
Here we compare the proposed \MILP formulation \eqref{mgclp-Int-MILP} with those in \smartcitet{Alvarez-Miranda2019}. 
For simplicity of illustration, we only compare \eqref{mgclp-Int-MILP} with the most related one, in which for each $j \in \J$, two continuous variables are used to measure the set functions  $\max_{(i,k) \in\CS} p_{ij}$ and $1 - \prod_{(i,k) \in \CS} (1-p_{ij})$ in problem \eqref{mgclp-Set-MINLP}.
Note that according to the results in \smartcitet{Alvarez-Miranda2019}, this formulation computationally \rev{dominates} the other three alternatives.

We first briefly review the \MILP formulation in \smartcitet{Alvarez-Miranda2019}. 
We introduce, for each $i \in \I $ and  $k \in [K]$, a binary variable $x_i^k$ 
denoting whether the $k$-th facility at location $i$ is opened.
Then the \MILP formulation proposed by \smartcitet{Alvarez-Miranda2019} can be written as 
\begin{align}\label{mgclp-Bin-MILP}\tag{MILP-B}
	\max_{x,\, \zeta, \, \eta} \ 
	\left\{
	\sum_{j \in \J} d_j \left(\theta \zeta_j + (1 -  \theta) \eta_j\right)\, : \,\eqref{cons:card-xik}\text{--}\eqref{cons:x-domain}  \right\}, 
\end{align}
where
\begin{subequations}
\begin{align}
	& \sum_{i \in \I } \sum_{k \in [K]} x_i^k \le K,
	\label{cons:card-xik} \\
	& x_i^k \ge x_i^{k+1}, \ \forall \ i \in \I , \ \forall \ k \in [K-1],
	\label{cons:sym-xik}\\
	& \zeta_j \le p_{\ell j} + \sum_{i\in \I} (p_{ij} - p_{\ell j})^+ x_i^1, \ \forall \ j \in \J,~ \forall \ \ell \in \{0\}\cup \I,
	\label{cons:zeta-linear-xik}\\
	& \eta_j \le \Phi_j(y^*) + \sum_{i \in \I} \rho_{ij}(y^*) \sum_{k = y_i^* + 1}^{K} x_i^k, \ 
	\forall \ j \in \J, \ \forall \ y^* \in \rev{\Y},
	\label{cons:eta-linear-xik}\\
	& x_i^k \in \{0,1\}, \ \forall \ i \in \I , \ \forall \ k \in [K]. 
	\label{cons:x-domain}
\end{align}
\end{subequations}
Here, \rev{we recall that $\Y=\{y \in  \rev{\mathbb{Z}^{|\I|}_{\geq 0}}\, : \, \sum_{i \in \I} y_i \leq K \}$}; 
for each $j \in \J$ and $i \in \I$, $\rho_{ij}(y^*) := \Phi_j(y^* + \be^i) - \Phi_j(y^*) 
= p_{ij} \prod_{\ell\in \I} (1-p_{\ell j})^{y^*_{\ell}}$ \rev{where}  $\be^i$ denotes the $i$-th unit vector of dimension $|\I|$;
constraint \eqref{cons:card-xik} ensures that at most $K$ facilities are open;
constraints \eqref{cons:sym-xik} require that if the $(k+1)$-th facility at a location is opened, then the first $k$ facilities in the same location must also be opened (which are used to exclude symmetric solutions with regard to co-location);
constraints \eqref{cons:zeta-linear-xik} are similar to those in \eqref{cons:zeta-linear-all} and are used to characterize $\zeta_j 
\leq \max_{(i,k) \in \CS} p_{ij}$;
constraints \eqref{cons:eta-linear-xik} are the so-called submodular inequalities corresponding to $\eta_j \leq 1 - \prod_{(i,k) \in \CS} (1-p_{ij})$. 
For more details of the above \MILP formulation, we refer to \smartcitet{Alvarez-Miranda2019}.

Although formulation \eqref{mgclp-Bin-MILP} involves {a huge number of constraints \eqref{cons:zeta-linear-xik} and \eqref{cons:eta-linear-xik}}, similar to formulation \eqref{mgclp-Int-MILP}, it can be solved by a \BnC framework where constraints \eqref{cons:zeta-linear-xik} and \eqref{cons:eta-linear-xik} are \rev{dynamically} separated.
However, the number of variables in the \MILP formulation \eqref{mgclp-Bin-MILP} of \smartcitet{Alvarez-Miranda2019} is  $ K|\I|+ 2|\J|$. 
This is 
one order of magnitude larger than that in the proposed formulation \eqref{mgclp-Int-MILP}, which is $ 2|\I|+ 2|\J|$. 
As such, it can be expected that 
\rev{formulation \eqref{mgclp-Bin-MILP} is much more computationally demanding than the proposed formulation \eqref{mgclp-Int-MILP}}
in a \BnC framework, especially for problems with a large $K$. 
In \cref{sect:numerical-experiments}, we will further demonstrate this by computational experiments.

%% file: linearization.tex
\section{Strong valid inequalities}\label{sect:valid-ineq}

Note that the derivation of the outer-approximation {inequalities} \eqref{cons:prod-outer-approx-explicit} does not exploit the integrality requirements of variables $y$ and other related constraints in \eqref{mgclp-Int-MINLP2}.
This disadvantage, however, may lead to weak outer-approximation inequalities \eqref{cons:prod-outer-approx-explicit} and a weak LP relaxation of problem \eqref{mgclp-Int-MILP}.
To strengthen the \LP relaxation of \eqref{mgclp-Int-MILP} and accelerate the convergence of the proposed \BnC algorithm, in this section, we develop two families of strong valid inequalities for problem \eqref{mgclp-Int-MINLP2} (and thus also for \eqref{mgclp-Int-MILP}) that explicitly consider other constraints (including the integrality constraints on variables $y$ and $z$). 
The two inequalities are developed by considering the set defined by \eqref{cons:zi-yi}--\eqref{cons:z-domain} and \eqref{cons:eta-nonlinear}, which can be presented 
as
\begin{equation}
		\X = \left\{ (\eta, y, z) \in \R \times \Z_{\rev{\ge 0}}^{|\I|} \times \{0,1\}^{|\I|} : 
		 \eta \le 1 - \prod_{i\in \I} (1-p_{i})^{y_i}, \ z_i \le y_i \le K z_i , \ \forall \ i \in \I
		\right\}.
\end{equation}

\subsection{Enhanced outer-approximation inequalities} 
\label{subsect:enhanved-oa-cuts}

Observe that $\X$ can be seen \rev{as} a restriction of \rev{$\U$}, defined in \eqref{Sdef}, that additionally enforces constraints $z_i \leq y_i \leq K z_i$ and $z_i \in \{0,1\}$ for all $i \in \I$. 
Thus, the outer-approximation \rev{inequalities} \eqref{cons:prod-outer-approx-explicit}, valid for \rev{$\U$}, \rev{are} also valid for $\X$.
In this subsection, we further strengthen the outer-approximation inequalities \eqref{cons:prod-outer-approx-explicit} using the constraints $z_i \leq y_i \leq K z_i$ and $(y_i,z_i) \in\mathbb{Z}_{\rev{\ge 0}}\times \{0,1\}$ for $i \in \I$ in $\X$.

We start with the following lemma stating the results on the bounds of the constant $c(y^*)$ and the coefficients $a_i(y^*)$ in inequality \eqref{cons:prod-outer-approx-explicit}, where $c(y^*)$ and $a_i(y^*)$ are defined in \eqref{def:c} and \eqref{def:ai}, respectively.
\begin{lemma}\label{obs:range-ai-c}
	Let $y^\ast \in \Z_{\rev{\ge 0}}^{|\I|}$. Then (i) $0 \le c(y^\ast) \le 1$ holds; and (ii) $a_i(y^\ast) \ge 0$ holds for all $i \in \I_{\Partial}$.
\end{lemma}
\begin{proof}
Observe that $c(y^\ast) \le 1$ and $a_i(y^\ast) \ge 0$ for $i \in \I_{\Partial}$ 
directly follow from the nonnegativity of $y^\ast$, $0 \leq p_i < 1$ for $i \in \I_{\Partial}$, and the definitions in \eqref{def:c} and \eqref{def:ai}.
In addition, using the concavity of $\hat{\Phi}(y):= 1- \prod_{i \in \I_{\Partial}} (1-p_i)^{y_i}$,
we obtain $\hat{\Phi}(y) \leq  \hat{\Phi}(y^\ast) + \sum_{i\in \I_{\Partial}} \frac{\partial \hat{\Phi}}{\partial y_i} (y^\ast) \cdot  (y_i - y_i^\ast)$.
Setting $y=\bzero$, where $\bzero$ denotes the all-zero vector, 
yields $0 \leq  \hat{\Phi}(y^\ast) + \sum_{i\in \I_{\Partial}} \frac{\partial \hat{\Phi}}{\partial y_i} (y^\ast) \cdot  (0 - y_i^\ast)=c(y^*)$, which completes the proof.
\end{proof}

Using \cref{obs:range-ai-c}, we can derive valid inequalities that are stronger than the outer-approximation inequalities in \eqref{cons:prod-outer-approx-explicit}.
Specifically, we can first strengthen the coefficients of variables $y_i$ that are larger than $1-c(y^*)$ in \eqref{cons:prod-outer-approx-explicit} to $1-c(y^*)$, and obtain the following inequality:
\begin{equation}\label{ineq:coef-str-step1}
	\eta \le c(y^\ast) + \sum_{i \in \I_{\Partial} \backslash \CL} a_i(y^\ast)y_i + (1-c(y^\ast)) \sum_{i \in \I_{\Full}\cup \CL} y_i,
\end{equation}
where $ \CL:= \{ i\in \I_{\Partial}: a_i(y^\ast) \geq 1 - c(y^\ast)\}$.
Indeed, the validity of \eqref{ineq:coef-str-step1} for $\X$ can be obtained by noting that for a given $(\eta, y, z)\in \X$, 
\begin{itemize}
	\item [(i)] if $y_i= 0$ for all $i \in \I_{\Full} \cup \CL$, then \eqref{ineq:coef-str-step1} reduces to \eqref{cons:prod-outer-approx-explicit} and thus \rev{the result} holds at $(\eta,y,z)$ (as \eqref{cons:prod-outer-approx-explicit} is valid for $\X$); and 
	\item [(ii)] if $y_{i_0}\geq 1$ for some $i_0\in \I_{\Full} \cup \CL$, then
	\begin{equation*}
		\begin{aligned}
	& c(y^\ast) + \sum_{i \in \I_{\Partial} \backslash \CL} a_i(y^\ast)y_i + (1-c(y^\ast)) \sum_{i \in \I_{\Full}\cup \CL} y_i \\
	& \qquad \qquad\qquad\quad\overset{(a)}{\geq} c(y^*) + (1-c(y^*)) y_{i_0} \geq  c(y^*) + 1-c(y^*)= 1 \overset{(b)}{\geq} \eta,
		\end{aligned} 
	\end{equation*}
	where (a) follows from $a_{i}(y^*) \geq 0$ for all $i \in \I_{\Partial} \backslash\CL$, $c(y^*) \leq 1$, and $y \in \mathbb{Z}_{\rev{\ge 0}}^{|\I|}$, and (b) follows from $\eta \leq 1 - \prod_{i\in \I} (1-p_{i})^{y_i}\leq 1$.
\end{itemize}

	We can further strengthen inequality \eqref{ineq:coef-str-step1} using constraints $z_i \leq y_i \leq K z_i$ for $i \in \I$, $y \in \Z_{\rev{\ge 0}}^{|\I|}$, and $z \in \{0,1\}^{|\I|}$  in $\X$.
	Indeed, these constraints imply that for all $i \in \I_{\Full}\cup\CL$, $z_i = 1$ holds if and only if $1 \leq y_i \leq K$ holds. 
	Consequently, using a similar discussion as in (i) and (ii), we can show that the following enhanced outer-approximation inequality
	\begin{align}\label{ineq:coef-str-step2}\tag{EOA}
		\eta \le c(y^\ast) 
		+ \sum_{i\in \I_{\Partial} \backslash \CL} a_{i}(y^\ast)y_i 
		+ (1-c(y^\ast))\sum_{i\in \I_{\Full} \cup \CL} z_i
	\end{align}
	 is valid for $\X$.
	Inequality \eqref{ineq:coef-str-step2} is stronger than inequality \eqref{ineq:coef-str-step1} as $y_i \ge z_i$ for all $i \in \I_{\Full} \cup \CL$ and $1-c(y^*) \geq 0$ (implied by \cref{obs:range-ai-c} (i)).
	

	\rev{An illustrative example for inequalities \eqref{ineq:coef-str-step2} to be strictly stronger than outer-approximation inequalities \eqref{cons:prod-outer-approx-explicit} is provided in Appendix \ref*{sect:appendix-EOA-example}.
	In addition, as enhanced versions of the outer-approximation inequalities, their separation can be conducted in a manner similar to that for inequalities \eqref{cons:prod-outer-approx-explicit}; see  \cref{subsec:MILP}.}

%% file: strong_valid_ineq_v2.tex
\subsection{Lifted subadditive inequalities}\label{subsect:strong-valid-ineqs}

In this subsection, we further develop another class of strong valid inequalities for $\conv(\X)$. 
In particular, we first derive a class of valid inequalities, called subadditive inequalities, for the restriction of $\X$: 
{\small
\begin{equation}\label{def:X0}
	\W = \left\{(\eta, y, z) \in \R \times \Z_{\rev{\ge 0}}^{|\I_{\Partial}|} \times \{0,1\}^{|\I_{\Partial}|}: \eta \le \Phipartial(y) = 1 - \prod_{i \in \I_{\Partial}}(1-p_i)^{y_i}, \ 
	z_i \le y_i \le K z_i, \ \forall \ i \in \I_{\Partial} \right\}, 
\end{equation}}obtained by setting $ z_i = 0$ and (thus) $y_i=0$ for \rev{all} $i \in \I_{\Full}$.
Then we lift them to obtain lifted  subadditive inequalities for $\conv(\X)$ and provide a mild condition for the derived inequalities to be facet-defining for $\conv(\X)$.
Finally, we discuss the separation procedure for the derived  inequalities.

\subsubsection{Subadditive inequalities for the restriction $\W$}

To derive valid inequalities for $\W$, let us consider the following exponential function:  
\begin{equation}\label{def:q}
	q(\delta) = 1 - e^{\delta}, ~ \text{where}~\delta \in \R_{\rev{\le 0}}.
\end{equation}
Clearly, $q$ is a concave function on $\R_{\rev{\le 0}}$ with $q(0)=0$. 
Thus, it follows from \citet[Theorem 7.2.5]{Hille1996} that
$q$ is subadditive on $\R_{\rev{\le 0}}$.
That is, 
\begin{lemma}\label{lem:q-subadditivity}
	$q(\delta_1 + \delta_2) \le q(\delta_1) + q(\delta_2)$ for all $\delta_1, \delta_2 \in \R_{\rev{\le 0}}$.
\end{lemma}
Note that $\Phi_{\text{P}}(y) = 1 - \prod_{i \in \I_{\Partial}}(1-p_i)^{y_i}= q\left(\sum_{i\in \I_{\Partial}}[\ln(1-p_i)] y_i\right)$.
The following lemma provides linear upper approximations for 
$q\left([\ln(1-p_i)]y_i\right)=1 - (1 - p_{i})^{y_{i}}$ with $(y_i, z_i) \in \mathbb{Z}_{\rev{\ge 0}}\times \{0,1\}$ and $z_i \leq y_i \leq Kz_i$, and 
$q\left([\ln(1-p_i)](y_i-z_i)\right)=1 - (1 - p_{i})^{y_{i}-z_i}$ with $(y_i, z_i) \in \mathbb{Z}_{\rev{\ge 0}}\times \{0,1\}$ and $z_i \leq y_i $, 
which will be useful in developing the coming valid inequalities for $\W$. 
The proof is provided in Appendix \ref*{sect:appendix-auxilary-subadditive}.
\begin{lemma}\label{lem:foundations}
	Let $i \in \I_{\Partial}$.
	The following statements hold:
	\begin{itemize}
		\item[(i)] 
		$1 - (1 - p_{i})^{y_{i}} \le h_{i, k}(y_i, z_i)$ holds for any $(y_i, z_i) \in \Z_{\rev{\ge 0}} \times \{0,1\}$ with $z_i \le y_i \le K z_i$, where $k \in [K-1]$ and
		\begin{equation}\label{def:h_k}
			\begin{aligned}
				h_{i,k}(y_i, z_i) & := \left(\left(1 - (1-p_i)^{k+1}\right) - \left(1 - (1-p_i)^{k}\right) \right) (y_i - k z_i) +  \left(1 - (1-p_i)^{k}\right) z_i\\
				& =p_i(1-p_i)^k y_i + \left(1  -(1-p_i)^k(k p_i+1)\right) z_i.
			\end{aligned}
		\end{equation}
		\item[(ii)] $1 - (1 - p_i)^{y_i - z_i} \le p_i (y_i - z_i)$ holds for any $(y_i, z_i) \in \Z_{\rev{\ge 0}} \times \{0,1\}$ with $z_i \le y_i $.
	\end{itemize}
\end{lemma}

Using the subadditivity of function $q$ in \cref{lem:q-subadditivity} and the linear upper approximations of $q([\ln(1-p_i)]y_i)$ and $q([\ln(1-p_i)](y_i - z_i))$ in \cref{lem:foundations}, 
 we can develop a class of subadditive inequalities for $\W$.

\begin{proposition}\label{thm:ineq-C-valid-for-W}
	Let $\C \subseteq \I_{\Partial}$ and $k_i \in [K-1]$ for $i \in \I_{\Partial}\backslash \C$.
	Then the subadditive inequality 
	\begin{equation}\label{ineq:goal-lift-and-proj-C}
		\begin{aligned}
			\eta 
			 \le 1 - \rev{p_{\C}} 
			+ \rev{p_{\C}} \cdot \left( \sum_{i \in \I_{\Partial} \backslash \C} h_{i,k_i}(y_i, z_i) + \sum_{i \in \C} p_i(y_i - z_i)\right)
				\end{aligned}
	\end{equation}
	is valid for $\W$.
    \rev{Here, $p_{\C} := \prod_{i \in \C} (1 - p_i)$ (when $\C=\varnothing$, $p_\C:=1$).}
\end{proposition}
\begin{myproof}
Let $(\eta, y, z) \in \W$.
Then,
\begin{equation}\label{ineq-C-step1}
	\begin{aligned}
		\eta  
		& \le 1 - \prod_{i \in \I_{\Partial}} (1 - p_i)^{y_i} 
		= 1 - \prod_{i \in \C} (1 - p_i) + \prod_{i \in \C} (1 - p_i) - \prod_{i \in \I_{\Partial}} (1 - p_i)^{y_i}\\
			& = 1 - \rev{p_{\C}} + \rev{p_{\C}} \cdot \left(1 - \bigg(\prod_{i \in \I_{\Partial} \backslash \C} (1 - p_i)^{y_i}\bigg) \bigg(\prod_{i \in \C} (1 - p_i)^{y_i - 1}\bigg)\right).
	\end{aligned}
\end{equation}
Observe that
\begin{equation}\label{ineq-C-step2}
	\begin{aligned}
		1 - \bigg(\prod_{i \in \I_{\Partial} \backslash \C} (1 - p_i)^{y_i}\bigg) \cdot \bigg(\prod_{i \in \C} (1 - p_i)^{y_i - 1}\bigg)
		& \overset{\text{(a)}}{\le} 1 - \bigg(\prod_{i \in \I_{\Partial} \backslash \C} (1 - p_i)^{y_i}\bigg) \cdot \bigg(\prod_{i \in \C} (1 - p_i)^{y_i - z_i}\bigg)\\
		& \overset{\text{(b)}}{=} q\left( \sum_{i \in \I_{\Partial} \backslash \C} [\ln(1 - p_i)] y_i + \sum_{i \in \C} [\ln(1 - p_i)] (y_i - z_i)\right)\\
		& \overset{\text{(c)}}{\le} \sum_{i \in \I_{\Partial} \backslash \C} q \left([\ln(1 - p_i)] y_i\right) + \sum_{i \in \C} q\left([\ln(1 - p_i)] (y_i - z_i)\right)\\
		& \overset{\text{(d)}}{\le} \sum_{i \in \I_{\Partial} \backslash\C} h_{i,k_i}(y_i, z_i) 
		+ \sum_{i \in \C} p_i(y_i - z_i),
	\end{aligned}
\end{equation}
where (a) holds as $p_i \in [0,1)$ and $z_i \le 1$ for $i \in \C$, 
(b) follows from the definition of $q$ in \eqref{def:q}, 
(c) follows from the subadditivity of function $q$ in \cref{lem:q-subadditivity}, 
and 
(d) follows from \cref{lem:foundations} and $(y_i,z_i) \in \Z_{\rev{\ge 0}} \times \{0,1\}$ with 
$z_i \le y_i \le K z_i$ for  $i \in \I_{\Partial}$.
Combining \eqref{ineq-C-step1} and \eqref{ineq-C-step2} yields the validity of inequality \eqref{ineq:goal-lift-and-proj-C}.
\end{myproof}

\subsubsection{Lifted subadditive inequalities for $\conv(\X)$}

Next, we derive strong valid inequalities for $\conv(\X)$ from inequalities \eqref{ineq:goal-lift-and-proj-C}.
In particular, we ``lift'' variables $\{z_i\}_{i \in \I_{\Full}}$ into subadditive inequality \eqref{ineq:goal-lift-and-proj-C}, obtaining the following valid inequality for $\conv(\X)$:   
\begin{align}\label{ineq:goal-lift-and-proj-C-lifted}\tag{LS}
		\eta 
		 \le 1 - \rev{p_{\C}} 
		+ \rev{p_{\C}} \cdot \left( \sum_{i \in \I_{\Partial} \backslash \C} h_{i,k_i}(y_i, z_i) + \sum_{i \in \C} p_i(y_i - z_i) + \sum_{i \in \I_{\Full}} z_i \right).
\end{align}
To see the validity, let $(\eta, y, z) \in \X$.
If $z_i = 0$ holds for all $i \in \I_{\Full}$, then inequality \eqref{ineq:goal-lift-and-proj-C-lifted} follows from the validity of inequality \eqref{ineq:goal-lift-and-proj-C} in \cref{thm:ineq-C-valid-for-W}. 
Otherwise, as $z_i \in \{0,1\}$ for $i \in \I_{\Full}$, we must have $\sum_{i \in \I_{\Full}} z_i \ge 1$, and thus 
{\footnotesize
\begin{equation*}
1 - \rev{p_{\C}} 
+ \rev{p_{\C}} \cdot \left( \sum_{i \in \I_{\Partial} \backslash \C} h_{i,k_i}(y_i, z_i) + \sum_{i \in \C} p_i(y_i - z_i) + \sum_{i \in \I_{\Full}} z_i \right)  
\overset{(a)}{\ge} 1 - \rev{p_{\C}} 
+ \rev{p_{\C}} \cdot \left (0 + 0 +  \sum_{i \in \I_{\Full}} z_i \right) \overset{(b)}{\geq} 1 - \rev{p_{\C}} + \rev{p_{\C}} 
= 1 \overset{(c)}{\geq} \eta,
\end{equation*}}%
where (a) follows from $h_{i,k_i}(y_i, z_i) \ge 1 - (1 - p_i)^{y_i} \ge 0$ for $i \in \I_{\Partial} \backslash \C$ (from \cref{lem:foundations} (i) and $(y_i, z_i) \in \Z_{\rev{\ge 0}} \times \{0,1\}$ with $z_i \le y_i \le K z_i$) 
and $p_i(y_i - z_i) \ge 0$ for $i \in \C$ (from $y_i \ge z_i$ and $p_i \geq 0$), 
(b) follows from $\sum_{i \in \I_{\Full}}z_i \geq 1$ {and $p_i \in [0,1)$ for $i \in \C$},
and (c) follows from $\eta \leq  1 - \prod_{i\in \I} (1-p_{i})^{y_i} \leq 1$.

{We demonstrate the strength of the proposed  inequality \eqref{ineq:goal-lift-and-proj-C-lifted} by showing that under a mild condition, \eqref{ineq:goal-lift-and-proj-C-lifted} is facet-defining for $\conv(\X)$.}
\rev{The proof is provided in Appendix \ref*{sect:appendix-LSI-facet-defining}.}
\begin{theorem}\label{thm:strong-valid-X}
	\rev{Let $\C \subseteq \I_{\Partial}$ and $k_i \in [K-1]$ for $i \in \I_{\Partial}\backslash \C$.}
    If $\I_{\Full} \neq \varnothing$, then inequality \eqref{ineq:goal-lift-and-proj-C-lifted} is facet-defining for $\conv(\X)$.
\end{theorem}

\subsubsection{Separation of the lifted subadditive inequalities \eqref{ineq:goal-lift-and-proj-C-lifted}}

In this subsection, we discuss the algorithm for separating the lifted subadditive {inequalities} \eqref{ineq:goal-lift-and-proj-C-lifted} in a \BnC context.
Given a point $(\eta^*, y^*, z^*) \in \R \times \R_{\rev{\ge 0}}^{|\I|} \times [0,1]^{|\I|}$ with $\eta^* \leq 1$ and $z_i^* \le y_i^* \le K z_i^*$ for $i \in \I$,
the separation problem of inequalities \eqref{ineq:goal-lift-and-proj-C-lifted} asks to find an inequality  violated by $(\eta^*, y^*, z^*)$ or prove that none exists.
Note that due to the exponential numbers of the subsets $\C \subseteq \I_{\Partial}$ and integer vectors $\{k_i\}_{i  \in \I_{\Partial} \backslash \C} \in {[K-1]}^{|\I_\Partial\backslash \C|}$, it is impractical to solve the separation problem by enumeration. 
To address this challenge, we first formulate it as the following optimization problem where  
the right-hand side of inequality \eqref{ineq:goal-lift-and-proj-C-lifted}  at $(\eta^*, y^*, z^*) $ is minimized:
\begin{align}\label{sep}\tag{SEP}
\min_{\substack{\C\subseteq \I_{\Partial},\\ \{k_i\}_{i  \in \I_{\Partial} \backslash \C} \in {[K-1]}^{|\I_\Partial\backslash \C|}}} 1 - \rev{p_{\C}} 
+ \rev{p_{\C}} \cdot \left(
\sum_{i \in \I_{\Partial} \backslash \C} h_{i, k_i}(y^*_i, z^*_i)  + 
\sum_{i \in \C} p_i(y^*_i - z^*_i) + 
\sum_{i \in \I_{\Full}} z^*_i \right).
\end{align}
In order to solve the separation problem \eqref{sep}, we first note that for a fixed subset $\C\subseteq \I_{\Partial}$, determining the $\{ k_i \}_{i \in \I_{\Partial} \backslash \C}$ to minimize the right-hand side of inequality \eqref{ineq:goal-lift-and-proj-C-lifted} at $(\eta^*, y^*, z^*) $ is easy.
\begin{lemma}\label{lem:min-hk}
	Given $i \in \I_{\Partial}$ and $(y_i^*, z_i^*) \in \R_{\rev{\ge 0}} \times [0,1]$  with $z_i^* \le y_i^* \le K z_i^*$, define
	\begin{equation}\label{def:min-hk-index}
		k_i^*:= \left\{ 
		\begin{aligned}
			& \left\lfloor \frac{y_i^*}{z_i^*}\right\rfloor, && ~\text{if}~ z_i^* > 0;\\
			& 1, && ~\text{otherwise}.
		\end{aligned}
		\right.
	\end{equation}
	Then 
	\begin{equation}\label{hopt}
		h_{i, k_i^*}(y_i^*, z_i^*) = \min \left\{h_{i, k}(y_i^*, z_i^*): k \in [K-1]\right\}\geq 0.
	\end{equation}
\end{lemma}

\begin{myproof}
If $z_i^* = 0$, then by $z_i^* \le y_i^* \le K z_i^*$, it follows $y_i^* = 0$, and therefore  $h_{i,k}(y_i^*, z_i^*) = 0$ holds  for all $k \in [K-1]$ and \eqref{hopt} follows.
Otherwise, for any $k \in \{2, \dots, K-1\}$, 
it follows 
\begin{equation}\label{tmp-hk-diff}
	\begin{aligned}
		h_{i, k}(y_i^*, z_i^*) - h_{i, k-1}(y_i^*,z_i^*) 
		& = p_i(1-p_i)^{k} y_i^* + \left(1 - (1-p_i)^{k}(kp_i+1)\right)z_i^*\\
		& \qquad - p_i(1-p_i)^{k-1} y_i^* - \left(1 - (1-p_i)^{k-1}((k-1)p_i+1)\right) z_i^*\\
		& = p_i^2 (1 - p_i)^{k-1} z_i^* \left(k - \frac{y_i^*}{z_i^*}\right).
	\end{aligned}
\end{equation}
Therefore, if $k \le \floor{\frac{y_i^*}{z_i^*}}$, then $h_{i, k}(y_i^*, z_i^*) \le h_{i, k-1}(y_i^*, z_i^*)$; 
otherwise, $k \ge \floor{\frac{y_i^*}{z_i^*}} + 1 $ and $h_{i, k}(y_i^*, z_i^*) \ge h_{i, k-1}(y_i^*, z_i^*)$.
This implies $h_{i, k_i^*}(y_i^*, z_i^*) = \min \left\{h_{i, k}(y_i^*, z_i^*): k \in [K-1]\right\}$, where $k_i^*$ is defined in \eqref{def:min-hk-index}.
Finally, it follows from $y_i^* \geq k_i^* z_i^*$, $0 \leq p_i < 1$, and \eqref{def:h_k} that 
\begin{equation*}
h_{i,k^*_i}(y^*_i, z^*_i) = \left(\left(1 - (1-p_i)^{k_i^*+1}\right) - \left(1 - (1-p_i)^{k_i^*}\right)\right) (y^*_i - k_i^* z_i) +  \left(1 - (1-p_i)^{k_i^*}\right) z^*_i\geq 0. 
\end{equation*}
\end{myproof}

\cref{lem:min-hk} allows us to set $k_i = k_i^*$ (defined in  \eqref{def:min-hk-index}) for $i \in \I_{\Partial} \backslash \C$ in the separation problem \eqref{sep}, obtaining 
\begin{align}\label{prob:sepa-C}\tag{SEP'}
	\min_{\C \subseteq \I_{\Partial}} ~ 
	\nu(\C):=1 - \rev{p_{\C}} + 
	\rev{p_{\C}} \cdot 
	\left(
	\sum_{i \in \I_{\Partial} \backslash \C} h_{i, k_i^*}(y_i^*, z_i^*)
	+ 
	\sum_{i \in \C} p_i (y_i^* - z_i^*)
	+ 
	\sum_{i \in \I_{\Full}} z_i^*
	\right), 
\end{align}
where only the subset $\C\subseteq \I_{\Partial}$ needs to be determined. 
Unfortunately, the following theorem shows the  intrinsic difficulty of finding an optimal solution for problem \eqref{prob:sepa-C}. 
The proof is provided in Appendix \ref*{sect:appendix-nphard}.
\begin{theorem}\label{thm:np-hard}
	Problem \eqref{prob:sepa-C} is NP-hard.
\end{theorem}

Due to the negative result in \cref{thm:np-hard}, in the following, we attempt to design a heuristic algorithm to solve the separation problem of the lifted subadditive inequalities \eqref{ineq:goal-lift-and-proj-C-lifted}.
To proceed, we first characterize the structure of subset $\C$ in violated inequalities \eqref{ineq:goal-lift-and-proj-C-lifted} with $k_i = k_i^*$ for $i \in \I_{\Partial} \backslash \C$. 
\rev{The proof is provided in Appendix \ref*{sect:appendix-sepa-C}.}


\begin{proposition}\label{thm:C-deterministic}
	Let $(\eta^*, y^*, z^*) \in \R \times \R_{\rev{\ge 0}}^{|\I|} \times [0,1]^{|\I|}$ with $\eta^* \leq 1$ and $z_i^* \le y_i^* \le K z_i^*$ for $i \in \I$,
	and $\{k_i^*\}_{i \in \I_{\Partial}}$ be defined in \eqref{def:min-hk-index}.
	Suppose that inequality \eqref{ineq:goal-lift-and-proj-C-lifted} defined by $\C$ and $\{k_i^*\}_{i \in \I_{\Partial}\backslash \C}$ is violated {by} $(\eta^*, y^*, z^*)$ by $\epsilon > 0$.
	The following statements hold. 
	\begin{itemize}
		\item [(i)] If $y_{i_0}^* = z_{i_0}^* = 1$ holds for some $i_0 \in \I_{\Partial} \backslash\C$, then inequality \eqref{ineq:goal-lift-and-proj-C-lifted} defined by $\C\cup \{i_0\}$ and $\{k_i^*\}_{i \in \I_{\Partial}\backslash (\C\cup \{i_0\})}$ is violated {by} $(\eta^*, y^*, z^*)$ by at least $\epsilon$.
		\item [(ii)] If $y_{i_0}^* = z_{i_0}^* = 0$ holds for some $i_0 \in \C$, then inequality \eqref{ineq:goal-lift-and-proj-C-lifted} defined by $\C\backslash \{i_0\}$ and $\{k_i^*\}_{i \in \I_{\Partial}\backslash (\C\backslash \{i_0\})}$ is violated {by} $(\eta^*, y^*, z^*)$ by at least $\epsilon$.
	\end{itemize}
\end{proposition}

Letting  $\I_{\Partial}^0 = \{i \in \I_{\Partial}: y_i^* = z_i^* = 0\}$ and  $\I_{\Partial}^1 = \{i \in \I_{\Partial}: y_i^* = z_i^* = 1\}$,
\cref{thm:C-deterministic} implies that in the separation problem of \eqref{ineq:goal-lift-and-proj-C-lifted},
determining whether an element in $\I_{\Partial}^0 \cup \I_{\Partial}^1$ belongs to $\C$ is easy; that is,
in order to find an inequality \eqref{ineq:goal-lift-and-proj-C-lifted} violated by $(\eta^*, y^*, z^*) \in \R \times \R_{\rev{\ge 0}}^{|\I|} \times [0,1]^{|\I|}$ with $\eta^* \leq 1$ and $z_i^* \le y_i^* \le K z_i^*$ for $i \in \I$,
it suffices to consider $\C \subseteq \I_{\Partial}$ satisfying $\I_{\Partial}^1 \subseteq \C$ and $\C\cap \I_{\Partial}^0=\varnothing$. 
Therefore,  we only need to 
determine whether an element in $\I_{\Partial}^r:=\I \backslash (\I_{\Partial}^0 \cup \I_{\Partial}^1)$ belongs to $\C$.
To achieve this and determine a high-quality solution $\C$ for problem \eqref{prob:sepa-C} (as to find a violated inequality \eqref{ineq:goal-lift-and-proj-C-lifted}), we use the following local search procedure.
Specifically, we first initialize $\C:=\I_{\Partial}^1$.  
Then, in each iteration, we attempt to find the best solution from the neighborhood $\mathfrak{N}(\C)= \{ \C' \, : \, \C' = \C \cup \{i\}~\text{for}~i \in \I_{\Partial}^r \backslash \C ~\text{or}~\C' = \C \backslash \{i\}~\text{for}~i \in \I_{\Partial}^r \cap \C\}$, i.e., 
\begin{equation}\label{localopt}
	\C^* \in \argmin \{ \nu(\C')\, : \, \C' \in \mathfrak{N}(\C)\}.
\end{equation}
If $\nu(\C^*) < \nu(\C)$, then a better solution $\C^*$ is found and we update $\C\leftarrow \C^*$. 
\rev{The procedure is repeated until no better solution is found
(the overall separation procedure is summarized in \cref{alg:sepa2} \rev{in Appendix \ref*{sect:appendix-alg-heur-sepa}}).
Note that in each iteration, either a strictly better solution $\C^* \in \mathfrak{N}(\C)$ is found or the algorithm terminates.
This, together with the fact that the number of candidate sets $\C$ is finite, 
implies that the algorithm will terminate after a finite number of iterations.
In addition, the main computational cost in each iteration is to compute $\nu(\C \cup \{i\})$ for $i \in \I_{\Partial}^r \backslash \C$ and  $\nu(\C \backslash \{i\})$ for $i \in \I_{\Partial}^r \cap \C$ with the complexity of $\CO(| \I_{\Partial}^r| |\I|)$, where $\CO(|\I|)$ is taken on evaluating the function value $\nu$.
However, by noting that (i) $\nu(\C)= 1- \rev{p}_\C \rev{\omega}_{\C}$, 
where $\rev{p}_{\C} = \prod_{i \in \C} (1 - p_i)$ and 
\begin{equation*}
	\omega_{\C} = 1  - \left(
	\sum_{i \in \I_{\Partial} \backslash \C} h_{i, k_i^*}(y_i^*, z_i^*)
	+ 
	\sum_{i \in \C} p_i (y_i^* - z_i^*)
	+ 
	\sum_{i \in \I_{\Full}} z_i^*
	\right), 
\end{equation*}
and
(ii) $ \rev{p}_{\C\cup \{i\}}$ and $\rev{\omega}_{\C \cup \{i\}}$ (respectively,  $\rev{p}_{\C\backslash\{i\}}$ and $\rev{\omega}_{\C\backslash \{i\}}$) can be updated from $\rev{p}_\C$ and $\rev{\omega}_{\C}$ in constant time, the iteration complexity of the above algorithm can be improved to $\CO(| \I_{\Partial}^r|)$.}
%

%% file: numerical_results.tex
\section{Computational results}
\label{sect:numerical-experiments}

In this section, we present the computational results to demonstrate the efficiency of the proposed \BnC algorithm based on the formulation \eqref{mgclp-Int-MILP}, and the effectiveness of the  proposed enhanced outer-approximation inequalities \eqref{ineq:coef-str-step2} and 
the lifted subadditive inequalities \eqref{ineq:goal-lift-and-proj-C-lifted}.
To do this, we first perform numerical experiments to \rev{compare the computational performance} of our proposed \BnC algorithm \rev{with} the state-of-the-art approach in \smartcitet{Alvarez-Miranda2019}.
Then, we present computational results to evaluate the performance effect of 
the proposed enhanced outer-approximation inequalities \eqref{ineq:coef-str-step2} and 
the lifted subadditive inequalities \eqref{ineq:goal-lift-and-proj-C-lifted} on the overall solution process of the proposed \BnC algorithm.

The proposed \BnC algorithm was implemented in Julia in 1.7.3 using \cplex 20.1.0.
We set parameters of \cplex to run the code with a time limit of 3600 seconds and a relative \MIP gap tolerance of $0\%$.
Except for the ones mentioned above, other parameters of \cplex were set to their default values.
All computational experiments were conducted on a cluster of Intel(R) Xeon(R) Gold 6140 CPU @ 2.30GHz computers.

In our experiments, we used a testbed of $240$ \MPCLP instances, investigated by \smartcitet{Alvarez-Miranda2019}, where $40$ of them were \rev{previously} investigated by \smartcitet{Berman2019}. 
These instances were constructed based on 40 $K$-median instances from the OR-library \citep{Beasley1990} with uniform customer demands and identical numbers of customers and potential facility locations.  
The instances have up to $900$ customers and potential facility locations, and $K$ \rev{ranges} from $5$ to $200$; see \cref{table:Bin-Int-summary}  for more details.
The linear decay function \citep{Berman2003} is used to evaluate the 
probability $p_{ij}$ of customer $j$ being covered by the facilities at location $i$, based on the distance {${D}_{ij}$} between customer $j$ and facility location $i$.
Specifically, given the minimum and maximum coverage distances $r \ge 0$ and $R > r$, the probability $p_{ij}$ of customer $j$ being covered by a facility at location $i$ is computed as:
\begin{equation*}
	p_{ij}= \left\{
	\begin{aligned}
		& 1, && \text{ if } D_{ij} \le r;\\
		& 1 - \frac{D_{ij} - r}{R - r},&& \text{ if } r < D_{ij} < R;\\
		& 0, && \text{ if } D_{ij} \ge R.
	\end{aligned}
	\right. 
\end{equation*}
By definition, a customer $j$ is considered to be (i) fully covered by a facility at location $i$ if $D_{ij} \leq r$, (ii) partially covered by a facility at location $i$ if $ r < D_{ij} < R$, and (iii) not covered by a facility at location $i$ if $D_{ij}\geq R$.
According to \smartcitet{Alvarez-Miranda2019}, 
the values $(r,R)$ are taken from $\{(5, 20), (10, 25)\}$, and
the parameter $\theta$ in \eqref{mgclp-Set-MINLP} is taken from $\{0.2, 0.5, 0.8\}$.
 
\subsection{Comparison with the state-of-the-art approach in \cite{Alvarez-Miranda2019}}\label{subsec:results-bin-int}

We first perform computational experiments to demonstrate the \rev{computational superiority} of our proposed \BnC algorithm over the state-of-the-art \BnC approach in \cite{Alvarez-Miranda2019}. 
In particular, we compare the performance of the following two settings:
\begin{itemize}
	\item \testInt: the proposed \BnC algorithm based on formulation \eqref{mgclp-Int-MILP} (with integer variables $\{y_i\}_{\rev{i \in \I}}$ being used to model the co-location of facilities) \rev{in which the submodular inequalities \eqref{cons:zeta-linear-all}, 
    the enhanced outer-approximation inequalities \eqref{ineq:coef-str-step2}, and the lifted subadditive inequalities \eqref{ineq:goal-lift-and-proj-C-lifted} are separated and added to the nodes of the search tree.}
	\item \testBin: the \BnC algorithm based on formulation \eqref{mgclp-Bin-MILP} of \cite{Alvarez-Miranda2019}  
	(with binary variables $\{x_i^k \}_{\rev{i \in \I, \, k \in [K]}}$ being used to model the co-location of facilities) \rev{in which inequalities \eqref{cons:zeta-linear-xik} and \eqref{cons:eta-linear-xik} are separated and added to the nodes of the search tree.}
\end{itemize}
Note that in order to perform a fair comparison, we re-implemented the \BnC approach of  \cite{Alvarez-Miranda2019} (i.e., \testBin) and report results obtained in our computational environment. 
Also note that in our implementation the primal heuristic algorithm and preprocessing techniques of \cite{Alvarez-Miranda2019} were included to find a high-quality solution and to simplify the problem formulation, respectively.

\begin{table}[t]
\centering
\scriptsize
\caption{Overall performance comparison of settings \testBin and \testInt.}
\label{table:Bin-Int-summary}
\tabcolsep=4.5pt
{
\input{table_BinInt}
}
\end{table}

\cref{table:Bin-Int-summary} summarizes the computational results of settings \testBin and \testInt, grouped by $(\V, \K)$.
For each row of \cref{table:Bin-Int-summary}, we report 
the number of the potential facility locations (\V) which is equal to the number of customers \rev{$|\J|$}, 
the number of facilities to be opened (\K), and 
the total number of instances per row (\tblndata). 
Under each setting, we report 
the average number of variables in formulations \eqref{mgclp-Int-MILP} or \eqref{mgclp-Bin-MILP} after preprocessing (\tblnvar), 
the number of solved instances (\tblnsol), 
the average CPU time in seconds (\tbltime), 
the average number of nodes (\tblnode), 
the average end gap returned by the \cplex (\tblgap), and 
the average \LP relaxation gap at the root node (\tblrgap), computed as $\frac{o_{\text{root}} - o^*}{o^*} \times 100\%$ where $o_{\text{root}}$ and $o^*$ are the \LP relaxation \rev{bounds} obtained at the root node and the optimal value (or the best incumbent) of the \MPCLP, respectively.
At the end of the table, we report the summarized results for all instances.
For fair comparison purpose, we also include the results of \testBin under column \testLiterature in \cref{table:Bin-Int-summary}, taken from \smartcitet{Alvarez-Miranda2019}.
From the results in \cref{table:Bin-Int-summary}, we observe that our implementation of \testBin is competitive with that in \smartcitet{Alvarez-Miranda2019} (where the differences may come from different computational environments).

From \cref{table:Bin-Int-summary}, we first observe that, as expected, the numbers of variables (\tblnvar) in the proposed formulation \eqref{mgclp-Int-MILP} (after preprocessing) are significantly smaller than those in formulation  \eqref{mgclp-Bin-MILP} (after preprocessing), especially for instances with a large $K$.
Let us take the instances with $\V=400$ and $K = 133$ for an illustration.
The average number of variables in the proposed formulation \eqref{mgclp-Int-MILP} is $1548$. This is one order of magnitude smaller than that in formulation \eqref{mgclp-Bin-MILP}, which is $40440$.
Second, 
the average \LP relaxation gap (\tblrgap) at the root node returned by \testInt is $0.60\%$, while that returned by  \testBin is $1.73\%$.
This shows that the \LP relaxation at the root node under setting \testInt is much stronger than that under  setting \testBin.
This improvement is mainly attributed to the incorporation of strong valid inequalities \eqref{ineq:coef-str-step2} and \eqref{ineq:goal-lift-and-proj-C-lifted} into \testInt, which effectively strengthens the \LP relaxation of  formulation \eqref{mgclp-Int-MILP};
see \cref{subsec:results-techniques} further ahead.
Due to the above two advantages, the proposed \testInt significantly outperforms \testBin.
Overall, \testInt can solve $205$ instances among the $240$ instances to optimality, whereas \testBin can only solve $153$ of them to optimality.
For instances that are solved by either \testBin or \testInt, the average CPU time and number of nodes decrease from $979.05$ seconds and $8526$ to $124.17$ seconds and $2995$, respectively; 
for instances that cannot be solved by at least one setting, the average end gap returned by \testInt is one order of magnitude smaller, reducing from $0.60\%$ to $0.04\%$.
\begin{figure}[tbp]
	\centering
	\subcaptionbox{\label{subfig:bin-int-time}}{\includegraphics[scale=.4]{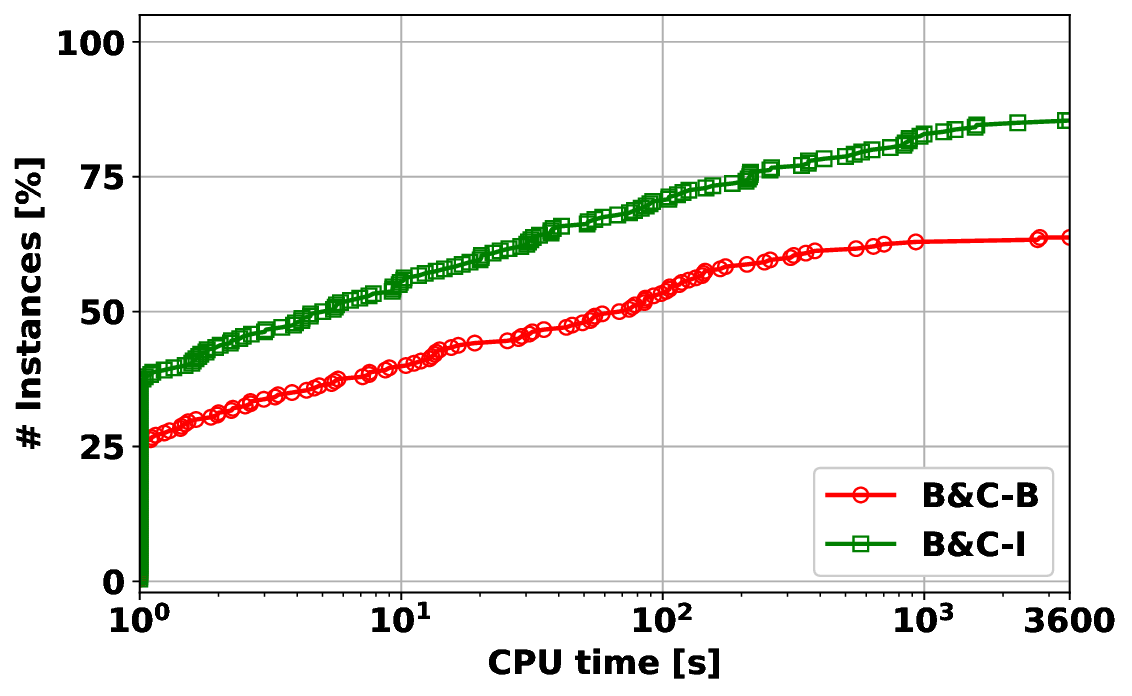}}
	\subcaptionbox{\label{subfig:bin-int-gap}}{\includegraphics[scale=.4]{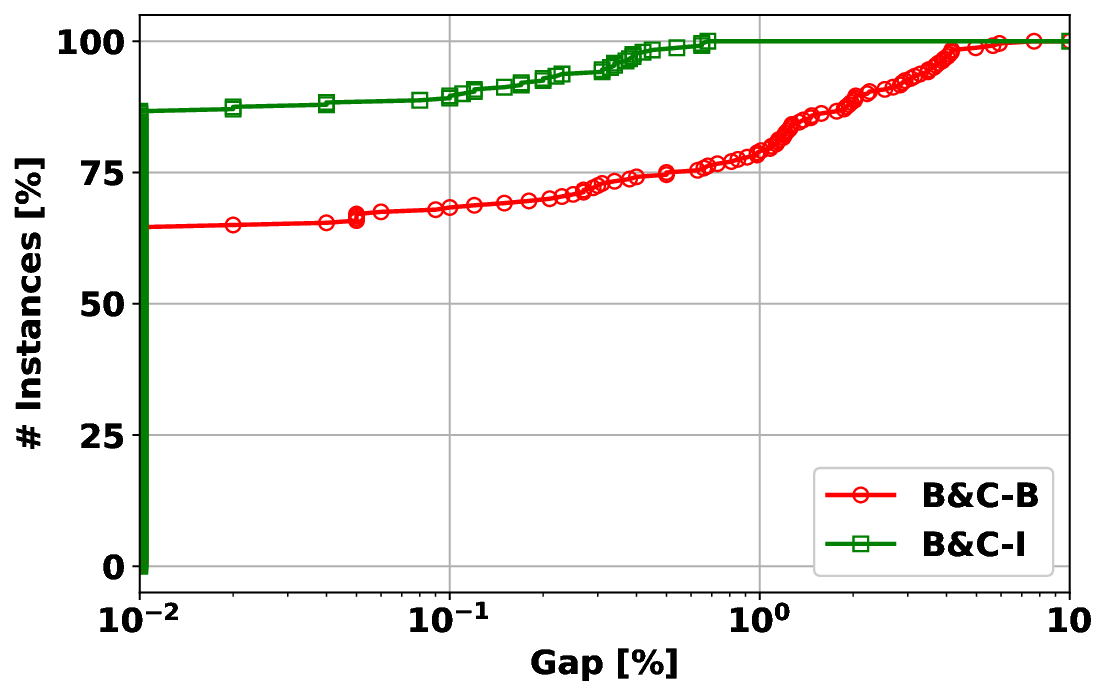}}
				\caption{Performance profiles of 
				 the CPU time and the end gap
				returned by settings \testBin and \testInt.}
	\label{fig:BIN-INT}
\end{figure}
The computational efficiency of the proposed \testInt over \testBin is more intuitively depicted in \cref{subfig:bin-int-time,subfig:bin-int-gap}, where we
plot the performance profiles of 
the CPU time and 
the end gap
under settings \testBin and \testInt, respectively.
In particular, from \cref{subfig:bin-int-time}, about $70\%$ of the instances can be solved by the proposed \testInt within $100$ seconds, whereas only about $55\%$ of the instances are solved by \testBin within the same period of time.
From \cref{subfig:bin-int-gap}, the end gaps returned by the proposed \testInt are all within $1\%$, while for
more than $20\%$ of instances, \testBin \ \rev{returns} an end gap larger than $1\%$.

It is worth noting that $148$ instances among the $240$ instances were solved to optimality by \smartcitet{Alvarez-Miranda2019}, while $205$ instances can be solved to optimality by {the proposed \BnC algorithm}. 
\rev{In \cref{table:unsolved} of Appendix \ref*{sect:newly-solved-instances}, we summarize the results for these $57$ newly solved instances including the optimal values of the instances under column \tblobj.}

\subsection{Performance effect of the enhanced outer-approximation inequalities \eqref{ineq:coef-str-step2} and lifted subadditive inequalities \eqref{ineq:goal-lift-and-proj-C-lifted}}
\label{subsec:results-techniques}

Next, we evaluate the performance effect of the proposed enhanced outer-approximation inequalities \eqref{ineq:coef-str-step2} and the lifted subadditive inequalities \eqref{ineq:goal-lift-and-proj-C-lifted} developed in \cref{sect:valid-ineq} on the proposed  \BnC algorithm based on formulation \eqref{mgclp-Int-MILP}.
To do this, we compare the following settings:
\begin{itemize}
	\item \testBasicInt: the vanilla version of \testInt, where both the enhanced outer-approximation inequalities \eqref{ineq:coef-str-step2} and the lifted subadditive inequalities \eqref{ineq:goal-lift-and-proj-C-lifted} were not implemented;
	\item \testBasicIntStr: setting \testBasicInt  with the enhanced outer-approximation inequalities \eqref{ineq:coef-str-step2};
	\item \testBasicIntVI: setting \testBasicInt with the lifted subadditive inequalities \eqref{ineq:goal-lift-and-proj-C-lifted};
	\item \testBasicIntStrVI: setting \testBasicInt with both the enhanced outer-approximation inequalities \eqref{ineq:coef-str-step2} and the lifted subadditive inequalities  \eqref{ineq:goal-lift-and-proj-C-lifted} (which is equivalent to setting \testInt).
\end{itemize}

\cref{table:VI-summary} summarizes the computational results.
First, we can observe from \cref{table:VI-summary} that the average \LP relaxation gap returned by \testBasicIntStr is smaller than that returned by \testBasicInt, which shows that the enhanced outer-approximation inequalities \eqref{ineq:coef-str-step2} can indeed strengthen the \LP relaxation of formulation \eqref{mgclp-Int-MILP}.
Therefore, equipped with \eqref{ineq:coef-str-step2}, \testBasicIntStr achieves a fairly better overall performance than \testBasicInt.
Overall, \testBasicIntStr can solve $6$ more instances to optimality within the time limit of $3600$ seconds than \testBasicInt; 
and for instances solved by \testBasicInt and \testBasicIntStr, the average CPU time and number of nodes decrease from $406.32$ seconds and $34054$ to $310.74$ seconds and $23937$, respectively.

Compared with the enhanced outer-approximation inequalities \eqref{ineq:coef-str-step2}, 
the lifted subadditive inequalities \eqref{ineq:goal-lift-and-proj-C-lifted} are even much stronger in terms of strengthening the \LP relaxation of formulation \eqref{mgclp-Int-MILP}; equipped with \eqref{ineq:goal-lift-and-proj-C-lifted}, the average \LP relaxation gap at the root node reduces from $1.10\%$ to $0.67\%$.
This is attributed to the favorable theoretical property of inequalities \eqref{ineq:goal-lift-and-proj-C-lifted} in \cref{thm:strong-valid-X}, which demonstrates that inequalities \eqref{ineq:goal-lift-and-proj-C-lifted} can be facet-defining for the substructure $\conv(\X)$ under mild conditions.
As a result, \testBasicIntVI achieves a much better overall performance than \testBasicInt.
Overall,  \testBasicIntVI can solve $13$ more instances to optimality than \testBasicInt within the time limit of $3600$ seconds;  and
the average CPU time and number of nodes returned by \testBasicIntVI are $219.21$ seconds and $6015$, respectively, while those returned by \testBasicInt are $406.32$ seconds and $34054$, respectively.

Note that the combination of the two proposed inequalities \eqref{ineq:coef-str-step2} and \eqref{ineq:goal-lift-and-proj-C-lifted} can even lead to a better \LP relaxation of formulation \eqref{mgclp-Int-MILP}, as demonstrated in column \tblrgap under setting \testBasicIntStrVI.
Therefore, compared with \testBasicInt, \testBasicIntStr, and \testBasicIntVI,  \testBasicIntStrVI with the two proposed inequalities \eqref{ineq:coef-str-step2} and \eqref{ineq:goal-lift-and-proj-C-lifted} can solve more instances to optimality within the time limit of $3600$ seconds, and the average CPU time and number of explored nodes are all much smaller.

\afterpage{
\begin{landscape}
\begin{table}[tbp]
	\caption{Performance comparison of settings \testBasicInt, \testBasicIntStr, \testBasicIntVI, and \testBasicIntStrVI.}
	\label{table:VI-summary}
	\tabcolsep=4.5pt
	\centering
	\scriptsize
	{
				\input{table_FourSetting}
	}
\end{table}
\end{landscape}}

In summary, our results show that both the enhanced outer-approximation inequalities \eqref{ineq:coef-str-step2} and the lifted subadditive inequalities  \eqref{ineq:goal-lift-and-proj-C-lifted}
can effectively strengthen the \LP relaxation of formulation \eqref{mgclp-Int-MILP}, leading to an overall better performance of the proposed \BnC algorithm for the \MPCLP.

%% file: table_BinInt.tex
\begin{tabular}{ccccrrrrrrrrrrrrrr}
\toprule
\multicolumn{1}{c}{\multirow{2}{*}{\id}}
 & \multicolumn{1}{c}{\multirow{2}{*}{\V}}
 & \multicolumn{1}{c}{\multirow{2}{*}{\K}}
 & \multicolumn{1}{c}{\multirow{2}{*}{\tblndata}}
 & \multicolumn{6}{c}{\testBin}
 & \multicolumn{6}{c}{\testInt}
 & \multicolumn{2}{c}{\testLiterature}
 \\
\cmidrule(r){5-10}
\cmidrule(r){11-16}
\cmidrule(r){17-18}
 &  &  & 
 &  \tblnvar &  \tblnsol &  \tbltime &  \tblnode &  \tblgap &  \tblrgap
 &  \tblnvar &  \tblnsol &  \tbltime &  \tblnode &  \tblgap &  \tblrgap
 &  \tblnsol &  \tbltime
 \\
\midrule
1      &     100 &       5 & 6      &     334 &       6 &    0.01 &       0 &    0.00 &    0.00 &     330 &       6 &    0.03 &       0 &    0.00 &    0.00 &       6 &    0.01 \\
2      &     100 &      10 & 6      &     424 &       6 &    0.04 &       0 &    0.00 &    0.00 &     327 &       6 &    0.04 &       0 &    0.00 &    0.00 &       6 &    0.03 \\
3      &     100 &      10 & 6      &     502 &       6 &    0.05 &       0 &    0.00 &    0.00 &     338 &       6 &    0.03 &       0 &    0.00 &    0.00 &       6 &    0.05 \\
4      &     100 &      20 & 6      &     441 &       6 &    0.04 &       0 &    0.00 &    0.00 &     337 &       6 &    0.02 &       0 &    0.00 &    0.00 &       6 &    0.05 \\
5      &     100 &      33 & 6      &    1052 &       6 &    0.13 &       0 &    0.00 &    0.25 &     331 &       6 &    0.04 &       0 &    0.00 &    0.00 &       6 &    0.14 \\
6      &     200 &       5 & 6      &     634 &       6 &    0.07 &       0 &    0.00 &    0.00 &     634 &       6 &    0.05 &       0 &    0.00 &    0.00 &       6 &    0.05 \\
7      &     200 &      10 & 6      &    1409 &       6 &    0.28 &       9 &    0.00 &    0.32 &     739 &       6 &    0.12 &       1 &    0.00 &    0.03 &       6 &    0.25 \\
8      &     200 &      20 & 6      &    1958 &       6 &    0.29 &       0 &    0.00 &    0.11 &     713 &       6 &    0.07 &       0 &    0.00 &    0.00 &       6 &    0.17 \\
9      &     200 &      40 & 6      &    3626 &       6 &    1.82 &      97 &    0.00 &    0.66 &     711 &       6 &    0.21 &      14 &    0.00 &    0.59 &       6 &    1.49 \\
10     &     200 &      67 & 6      &    8211 &       6 &   41.83 &    1246 &    0.00 &    1.01 &     728 &       6 &    1.10 &     209 &    0.00 &    0.54 &       6 &   26.21 \\
11     &     300 &       5 & 6      &    1427 &       6 &    0.50 &      11 &    0.00 &    0.23 &    1130 &       6 &    0.16 &       0 &    0.00 &    0.00 &       6 &    0.41 \\
12     &     300 &      10 & 6      &    1575 &       6 &    0.75 &      22 &    0.00 &    0.38 &    1136 &       6 &    0.30 &       2 &    0.00 &    0.02 &       6 &    0.58 \\
13     &     300 &      30 & 6      &    5953 &       6 &   23.11 &     931 &    0.00 &    1.80 &    1142 &       6 &    1.19 &      64 &    0.00 &    0.70 &       6 &   15.30 \\
14     &     300 &      60 & 6      &   11030 &       4 & 1214.26 &   39400 &    0.03 &    0.94 &    1133 &       6 &    2.70 &     162 &    0.00 &    0.32 &       6 &  368.60 \\
15     &     300 &     100 & 6      &   20971 &       4 & 1271.48 &   12295 &    0.28 &    1.28 &    1152 &       6 &  158.13 &   10714 &    0.00 &    0.72 &       4 & 1342.14 \\
16     &     400 &       5 & 6      &    1837 &       6 &    2.63 &      34 &    0.00 &    2.23 &    1552 &       6 &    1.26 &       9 &    0.00 &    0.35 &       6 &    2.78 \\
17     &     400 &      10 & 6      &    3157 &       6 &   21.84 &     710 &    0.00 &    2.00 &    1557 &       6 &    3.14 &      47 &    0.00 &    0.53 &       6 &   38.22 \\
18     &     400 &      40 & 6      &   11377 &       3 & 1846.42 &   46363 &    0.64 &    1.94 &    1557 &       6 &  161.62 &    3305 &    0.00 &    0.68 &       4 & 1692.33 \\
19     &     400 &      80 & 6      &   24630 &       3 & 1859.80 &   15116 &    0.57 &    1.52 &    1568 &       6 &  322.77 &   10615 &    0.00 &    0.88 &       3 & 1959.14 \\
20     &     400 &     133 & 6      &   40440 &       3 & 1468.23 &   10381 &    0.55 &    1.42 &    1548 &       5 &  192.41 &    9342 &    0.06 &    1.08 &       3 & 1514.09 \\
21     &     500 &       5 & 6      &    2351 &       6 &    3.98 &      24 &    0.00 &    0.80 &    1964 &       6 &    0.85 &       0 &    0.00 &    0.01 &       6 &    4.12 \\
22     &     500 &      10 & 6      &    3806 &       6 &   83.75 &    1910 &    0.00 &    3.17 &    1977 &       6 &   20.70 &     156 &    0.00 &    1.01 &       6 &  319.84 \\
23     &     500 &      50 & 6      &   21022 &       2 & 1890.49 &   17459 &    1.45 &    2.34 &    1970 &       4 &  641.81 &    9242 &    0.18 &    1.00 &       2 & 2068.37 \\
24     &     500 &     100 & 6      &   42818 &       1 & 2710.71 &   24952 &    0.89 &    1.36 &    1971 &       4 &  160.20 &    6291 &    0.10 &    0.72 &       2 & 1972.34 \\
25     &     500 &     167 & 6      &   73696 &       3 & 1511.16 &   10264 &    0.40 &    0.99 &    1973 &       5 &   83.80 &    3743 &    0.05 &    0.78 &       3 & 2024.92 \\
26     &     600 &       5 & 6      &    3254 &       6 &  497.41 &    5024 &    0.00 &    4.75 &    2382 &       6 &   12.33 &      65 &    0.00 &    1.08 &       5 &  788.28 \\
27     &     600 &      10 & 6      &    6082 &       2 & 2439.98 &   25822 &    1.42 &    4.21 &    2384 &       6 &  219.00 &    4069 &    0.00 &    1.13 &       2 & 2558.69 \\
28     &     600 &      60 & 6      &   33315 &       0 &     \TL &   41015 &    1.51 &    2.03 &    2380 &       2 &  373.75 &   17106 &    0.22 &    0.89 &       0 &     \TL \\
29     &     600 &     120 & 6      &   66327 &       1 & 2457.91 &   18901 &    0.79 &    1.21 &    2380 &       3 &  188.51 &    9933 &    0.14 &    0.73 &       1 & 2889.76 \\
30     &     600 &     200 & 6      &  109899 &       4 &  175.05 &    1239 &    0.18 &    0.40 &    2377 &       4 &    6.92 &     342 &    0.04 &    0.33 &       3 &  915.25 \\
31     &     700 &       5 & 6      &    3881 &       4 & 1689.71 &    8974 &    1.04 &    5.51 &    2790 &       6 &   94.98 &     271 &    0.00 &    2.07 &       4 & 1831.06 \\
32     &     700 &      10 & 6      &    6798 &       2 & 2447.76 &   18675 &    1.56 &    4.17 &    2790 &       6 &  150.04 &     493 &    0.00 &    0.67 &       2 & 2641.14 \\
33     &     700 &      70 & 6      &   46151 &       0 &     \TL &   20646 &    1.25 &    1.62 &    2789 &       2 &  573.26 &   25414 &    0.14 &    0.67 &       0 &     \TL \\
34     &     700 &     140 & 6      &   93486 &       4 &  245.91 &    1298 &    0.32 &    0.60 &    2778 &       4 &    7.21 &      59 &    0.06 &    0.30 &       0 &     \TL \\
35     &     800 &       5 & 6      &    4748 &       3 & 1929.45 &    4944 &    1.45 &    4.73 &    3191 &       6 &   30.27 &      79 &    0.00 &    1.17 &       2 & 2979.34 \\
36     &     800 &      10 & 6      &    7924 &       0 &     \TL &   20558 &    3.12 &    4.81 &    3191 &       3 & 1599.98 &    7192 &    0.19 &    1.69 &       0 &     \TL \\
37     &     800 &      80 & 6      &   60533 &       0 &     \TL &   21800 &    0.93 &    1.21 &    3188 &       2 &  859.75 &   36339 &    0.12 &    0.55 &       0 &     \TL \\
38     &     900 &       5 & 6      &    5688 &       2 & 2613.44 &    3624 &    3.03 &    5.50 &    3594 &       6 &  269.99 &     619 &    0.00 &    1.66 &       0 &     \TL \\
39     &     900 &      10 & 6      &    9956 &       0 &     \TL &    7578 &    2.17 &    3.05 &    3595 &       3 &  470.38 &    3062 &    0.09 &    0.77 &       0 &     \TL \\
40     &     900 &      90 & 6      &   79479 &       0 &     \TL &   22683 &    0.56 &    0.68 &    3595 &       2 &  856.29 &   40644 &    0.06 &    0.34 &       0 &     \TL \\
\multicolumn{3}{l}{Sol.} & 240 & 	 &     153 &   &   &    &  & 	 &     205 &   &    &  & &     148 & \\
\multicolumn{3}{l}{Aver.} &  & 	 &     &  979.05 &    8526 &    0.60 &    1.73 & 	 &     &  124.17 &    2995 &    0.04 &    0.60 &     & 1131.76 \\
\bottomrule
\end{tabular}

%% file: table_FourSetting.tex
\begin{tabular}{ccccrrrrrrrrrrrrrrrrrrrr}
\toprule
\multicolumn{1}{c}{\multirow{2}{*}{\id}}
 & \multicolumn{1}{c}{\multirow{2}{*}{\V}}
 & \multicolumn{1}{c}{\multirow{2}{*}{\K}}
 & \multicolumn{1}{c}{\multirow{2}{*}{\tblndata}}
 & \multicolumn{5}{c}{\testBasicInt}
 & \multicolumn{5}{c}{\testBasicIntStr}
 & \multicolumn{5}{c}{\testBasicIntVI}
 & \multicolumn{5}{c}{\testBasicIntStrVI}
 \\
\cmidrule(r){5-9}
\cmidrule(r){10-14}
\cmidrule(r){15-19}
\cmidrule(r){20-24}
 &  &  & 
 &  \tblnsol &  \tbltime &  \tblnode &  \tblgap &  \tblrgap
 &  \tblnsol &  \tbltime &  \tblnode &  \tblgap &  \tblrgap
 &  \tblnsol &  \tbltime &  \tblnode &  \tblgap &  \tblrgap
 &  \tblnsol &  \tbltime &  \tblnode &  \tblgap &  \tblrgap
 \\
\midrule
1      &     100 &       5 & 6      &       6 &    0.02 &       0 &    0.00 &    0.00 &       6 &    0.02 &       0 &    0.00 &    0.00 &       6 &    0.02 &       0 &    0.00 &    0.00 &       6 &    0.03 &       0 &    0.00 &    0.00 \\
2      &     100 &      10 & 6      &       6 &    0.03 &       0 &    0.00 &    0.00 &       6 &    0.03 &       0 &    0.00 &    0.00 &       6 &    0.03 &       0 &    0.00 &    0.00 &       6 &    0.04 &       0 &    0.00 &    0.00 \\
3      &     100 &      10 & 6      &       6 &    0.06 &      18 &    0.00 &    0.24 &       6 &    0.06 &      27 &    0.00 &    0.19 &       6 &    0.03 &       0 &    0.00 &    0.00 &       6 &    0.03 &       0 &    0.00 &    0.00 \\
4      &     100 &      20 & 6      &       6 &    0.03 &       0 &    0.00 &    0.00 &       6 &    0.03 &       0 &    0.00 &    0.00 &       6 &    0.03 &       0 &    0.00 &    0.00 &       6 &    0.02 &       0 &    0.00 &    0.00 \\
5      &     100 &      33 & 6      &       6 &    0.16 &     167 &    0.00 &    0.51 &       6 &    0.04 &       5 &    0.00 &    0.51 &       6 &    0.04 &       0 &    0.00 &    0.24 &       6 &    0.04 &       0 &    0.00 &    0.00 \\
6      &     200 &       5 & 6      &       6 &    0.05 &       0 &    0.00 &    0.00 &       6 &    0.05 &       0 &    0.00 &    0.00 &       6 &    0.05 &       0 &    0.00 &    0.00 &       6 &    0.05 &       0 &    0.00 &    0.00 \\
7      &     200 &      10 & 6      &       6 &    0.32 &      82 &    0.00 &    0.53 &       6 &    0.16 &      25 &    0.00 &    0.43 &       6 &    0.12 &       1 &    0.00 &    0.12 &       6 &    0.12 &       1 &    0.00 &    0.03 \\
8      &     200 &      20 & 6      &       6 &    0.30 &      99 &    0.00 &    0.41 &       6 &    0.17 &      45 &    0.00 &    0.37 &       6 &    0.07 &       0 &    0.00 &    0.00 &       6 &    0.07 &       0 &    0.00 &    0.00 \\
9      &     200 &      40 & 6      &       6 &    1.36 &     463 &    0.00 &    0.96 &       6 &    1.27 &     581 &    0.00 &    0.82 &       6 &    0.23 &      17 &    0.00 &    0.38 &       6 &    0.21 &      14 &    0.00 &    0.59 \\
10     &     200 &      67 & 6      &       6 &    5.75 &    2155 &    0.00 &    1.03 &       6 &    3.50 &    1109 &    0.00 &    0.98 &       6 &    1.20 &     199 &    0.00 &    0.55 &       6 &    1.10 &     209 &    0.00 &    0.54 \\
11     &     300 &       5 & 6      &       6 &    0.28 &      18 &    0.00 &    0.39 &       6 &    0.25 &      13 &    0.00 &    0.32 &       6 &    0.16 &       0 &    0.00 &    0.00 &       6 &    0.16 &       0 &    0.00 &    0.00 \\
12     &     300 &      10 & 6      &       6 &    1.21 &     207 &    0.00 &    0.69 &       6 &    0.84 &      95 &    0.00 &    0.82 &       6 &    0.31 &       2 &    0.00 &    0.03 &       6 &    0.30 &       2 &    0.00 &    0.02 \\
13     &     300 &      30 & 6      &       6 &   10.92 &    1865 &    0.00 &    1.68 &       6 &    6.51 &    1143 &    0.00 &    1.54 &       6 &    1.66 &      97 &    0.00 &    0.79 &       6 &    1.19 &      64 &    0.00 &    0.70 \\
14     &     300 &      60 & 6      &       6 &   17.92 &    3362 &    0.00 &    1.15 &       6 &   11.07 &    2018 &    0.00 &    0.73 &       6 &    3.63 &     164 &    0.00 &    0.31 &       6 &    2.70 &     162 &    0.00 &    0.32 \\
15     &     300 &     100 & 6      &       4 & 1248.85 &  298552 &    0.05 &    1.33 &       5 &  923.58 &  177106 &    0.02 &    1.20 &       6 &  335.68 &   23793 &    0.00 &    0.73 &       6 &  158.13 &   10714 &    0.00 &    0.72 \\
16     &     400 &       5 & 6      &       6 &    4.32 &     272 &    0.00 &    2.62 &       6 &    2.09 &     129 &    0.00 &    1.59 &       6 &    1.20 &      11 &    0.00 &    0.44 &       6 &    1.26 &       9 &    0.00 &    0.35 \\
17     &     400 &      10 & 6      &       6 &   36.15 &    2274 &    0.00 &    2.26 &       6 &   17.89 &    1180 &    0.00 &    1.60 &       6 &    5.45 &      73 &    0.00 &    0.55 &       6 &    3.14 &      47 &    0.00 &    0.53 \\
18     &     400 &      40 & 6      &       4 & 1306.30 &  133009 &    0.18 &    1.42 &       4 & 1257.72 &  125900 &    0.14 &    1.19 &       6 &  204.93 &    4245 &    0.00 &    0.74 &       6 &  161.62 &    3305 &    0.00 &    0.68 \\
19     &     400 &      80 & 6      &       4 & 1403.64 &  150551 &    0.18 &    1.54 &       4 & 1233.41 &  114785 &    0.08 &    1.24 &       5 &  989.91 &   36115 &    0.01 &    1.09 &       6 &  322.77 &   10615 &    0.00 &    0.88 \\
20     &     400 &     133 & 6      &       3 & 1459.29 &  170378 &    0.15 &    1.43 &       3 & 1444.90 &  152178 &    0.10 &    1.32 &       5 &  684.58 &   34933 &    0.06 &    1.10 &       5 &  192.41 &    9342 &    0.06 &    1.08 \\
21     &     500 &       5 & 6      &       6 &    1.03 &      17 &    0.00 &    0.31 &       6 &    0.94 &      12 &    0.00 &    0.28 &       6 &    0.90 &       1 &    0.00 &    0.01 &       6 &    0.85 &       0 &    0.00 &    0.01 \\
22     &     500 &      10 & 6      &       6 &  149.55 &    3381 &    0.00 &    2.56 &       6 &   71.53 &    1907 &    0.00 &    1.77 &       6 &   28.16 &     179 &    0.00 &    1.34 &       6 &   20.70 &     156 &    0.00 &    1.01 \\
23     &     500 &      50 & 6      &       2 & 1849.79 &  124580 &    0.45 &    1.73 &       3 & 1202.68 &   51015 &    0.30 &    1.41 &       3 & 1029.58 &   10686 &    0.27 &    1.26 &       4 &  641.81 &    9242 &    0.18 &    1.00 \\
24     &     500 &     100 & 6      &       2 & 1810.83 &  141452 &    0.32 &    1.01 &       3 & 1102.36 &   95127 &    0.13 &    0.87 &       4 &  256.60 &    6840 &    0.23 &    0.79 &       4 &  160.20 &    6291 &    0.10 &    0.72 \\
25     &     500 &     167 & 6      &       4 &  757.22 &   64705 &    0.16 &    1.02 &       4 &  730.45 &   63608 &    0.08 &    0.96 &       5 &  189.27 &   12210 &    0.10 &    0.80 &       5 &   83.80 &    3743 &    0.05 &    0.78 \\
26     &     600 &       5 & 6      &       6 &   22.90 &     362 &    0.00 &    1.88 &       6 &   11.70 &     175 &    0.00 &    1.61 &       6 &   15.71 &      84 &    0.00 &    1.19 &       6 &   12.33 &      65 &    0.00 &    1.08 \\
27     &     600 &      10 & 6      &       6 &  576.45 &   23942 &    0.00 &    1.73 &       6 &  242.23 &    8042 &    0.00 &    1.48 &       6 &  381.48 &   10464 &    0.00 &    1.17 &       6 &  219.00 &    4069 &    0.00 &    1.13 \\
28     &     600 &      60 & 6      &       1 & 1855.16 &  113154 &    0.38 &    1.39 &       2 & 1128.42 &   68980 &    0.28 &    1.16 &       2 & 1167.79 &   68050 &    0.30 &    1.04 &       2 &  373.75 &   17106 &    0.22 &    0.89 \\
29     &     600 &     120 & 6      &       2 & 1336.50 &   92282 &    0.39 &    0.93 &       3 &  145.27 &    9532 &    0.16 &    0.85 &       3 &  550.99 &   26907 &    0.23 &    0.78 &       3 &  188.51 &    9933 &    0.14 &    0.73 \\
30     &     600 &     200 & 6      &       4 &   59.70 &    5586 &    0.10 &    0.42 &       4 &   18.86 &    1963 &    0.07 &    0.40 &       4 &    9.88 &     341 &    0.09 &    0.34 &       4 &    6.92 &     342 &    0.04 &    0.33 \\
31     &     700 &       5 & 6      &       6 &  127.67 &    1933 &    0.00 &    2.50 &       6 &   88.46 &    1245 &    0.00 &    2.44 &       6 &  157.44 &     325 &    0.00 &    2.24 &       6 &   94.98 &     271 &    0.00 &    2.07 \\
32     &     700 &      10 & 6      &       6 &  217.72 &    2610 &    0.00 &    1.20 &       6 &  254.37 &    8149 &    0.00 &    0.88 &       6 &  310.54 &    5837 &    0.00 &    0.77 &       6 &  150.04 &     493 &    0.00 &    0.67 \\
33     &     700 &      70 & 6      &       2 & 1617.30 &   91188 &    0.23 &    0.90 &       2 & 1292.93 &   74869 &    0.21 &    0.81 &       2 & 1223.19 &   60616 &    0.18 &    0.91 &       2 &  573.26 &   25414 &    0.14 &    0.67 \\
34     &     700 &     140 & 6      &       4 &   14.48 &     798 &    0.14 &    0.47 &       4 &    9.89 &     783 &    0.06 &    0.41 &       4 &   25.17 &     799 &    0.10 &    0.35 &       4 &    7.21 &      59 &    0.06 &    0.30 \\
35     &     800 &       5 & 6      &       6 &   28.91 &     183 &    0.00 &    1.31 &       6 &   20.04 &     121 &    0.00 &    1.27 &       6 &   42.12 &      96 &    0.00 &    1.25 &       6 &   30.27 &      79 &    0.00 &    1.17 \\
36     &     800 &      10 & 6      &       1 & 2503.83 &   23397 &    0.40 &    2.19 &       2 & 2400.32 &   24664 &    0.27 &    1.90 &       1 & 2533.96 &   17082 &    0.37 &    1.90 &       3 & 1599.98 &    7192 &    0.19 &    1.69 \\
37     &     800 &      80 & 6      &       2 & 1650.91 &   65238 &    0.19 &    0.68 &       2 & 1442.35 &   66929 &    0.14 &    0.60 &       2 &  335.12 &    5713 &    0.18 &    0.54 &       2 &  859.75 &   36339 &    0.12 &    0.55 \\
38     &     900 &       5 & 6      &       6 &  275.06 &    1925 &    0.00 &    2.11 &       6 &  179.73 &     894 &    0.00 &    1.74 &       6 &  562.94 &    1716 &    0.00 &    1.75 &       6 &  269.99 &     619 &    0.00 &    1.66 \\
39     &     900 &      10 & 6      &       3 &  647.78 &    8812 &    0.11 &    0.90 &       3 &  593.45 &    8231 &    0.11 &    0.88 &       3 &  609.61 &    4665 &    0.12 &    0.87 &       3 &  470.38 &    3062 &    0.09 &    0.77 \\
40     &     900 &      90 & 6      &       2 &  471.00 &   19361 &    0.11 &    0.44 &       2 &  189.55 &    7753 &    0.07 &    0.35 &       2 &  241.46 &    3794 &    0.09 &    0.37 &       2 &  856.29 &   40644 &    0.06 &    0.34 \\
\multicolumn{3}{l}{Sol.} &  240 &     188 &   &   &    &    &     194 &  &&    &     &     201 &  &    &  &   &     205 &   &     &    &    \\
\multicolumn{3}{l}{Aver.} &  &     &  406.32 &   34054 &    0.09 &    1.10 &     &  310.74 &   23937 &    0.06 &    0.92 &     &  219.21 &    6015 &    0.06 &    0.67 &     &  124.17 &    2995 &    0.04 &    0.60 \\
\bottomrule
\end{tabular}

%% file: conclusion.tex
\section{Conclusions}
\label{sect:conclusion}

In this paper, we have investigated the \MPCLP in a joint probabilistic coverage setting, and developed an efficient \rev{\LP-based} \BnC algorithm  for it. 
The proposed \BnC algorithm is based on an \MILP reformulation that incorporates submodular and outer-approximation inequalities, which are separated at the nodes of the search tree.
Moreover, to speed up the convergence of the proposed \BnC algorithm, we developed two families of strong valid inequalities, called enhanced outer-approximation and lifted subadditive inequalities, to strengthen the \LP relaxation of the underlying formulation.
Two key features of the proposed \BnC algorithm, which make it particularly suitable for solving large-scale \MPCLP{s}, are:
(i) it is built on a light-weight mathematical formulation whose number of variables \rev{grows linearly} with the number of facility locations and customers and is one order of magnitude smaller than the underlying formulation in the  state-of-the-art \BnC algorithm of \cite{Alvarez-Miranda2019},
(ii) it is equipped with the two families of strong valid inequalities that can significantly strengthen the \LP relaxation of the formulation, enabling a quick convergence.
Computational experiments on a testbed of $240$ benchmark instances \rev{from the literature} demonstrated that, thanks to the small problem size and the strong \LP relaxation of the underlying formulation, 
the proposed \BnC algorithm significantly outperforms the state-of-the-art approach 
in terms of running time, \rev{the} number of nodes in the search tree, and \rev{the} number of solved instances.
In particular, using the proposed \BnC algorithm, we are able to provide optimal solutions for $57$ previously unsolved benchmark instances within a time limit of one hour.

%% file: sepa_submodular.tex
\section{Separation of inequalities \eqref{cons:zeta-linear}}
\label{appendix:separation-zeta-linear}

Here, we consider the separation of inequalities \eqref{cons:zeta-linear}. 
That is, for $(\zeta^\ast, z^\ast) \in \R \times [0,1]^{|\I|}$, among the $|\I|+1$ inequalities in \eqref{cons:zeta-linear}, we either find a violated inequality or prove that none exists.
Without loss of generality, we assume that $\I = \{1, \dots, |\I|\}$ and $0 = p_0 \le p_1 \le \cdots \le p_{|\I|}$. 
Let $F_{\ell}(z^\ast)$,  $\ell \in \{0\} \cup \I$, be  the right-hand side of inequalities \eqref{cons:zeta-linear} at $(\zeta^\ast, z^\ast)$, i.e.,  $F_{\ell}(z^\ast) := p_{\ell} + \sum_{i \in \I} (p_i - p_{\ell})^+ z^\ast_i
= p_{\ell} + \sum_{i = \ell + 1}^{|\I|} (p_i - p_{\ell}) z^\ast_i$.
The separation problem is equivalent to determining $\ell^* \in \argmin_{\ell  \in \{0\}\cup \I} F_{\ell}(z^\ast)$.
Observe that for any $\ell \in \{1, \dots, |\I|\}$, it follows: 
\begin{equation*}\label{diff-Frz}
	F_{\ell}(z^\ast) - F_{\ell -1}(z^\ast)
	=p_{\ell}+ \sum_{i = \ell + 1}^{|\I|} (p_i - p_{\ell}) z^\ast_i
	-p_{\ell -1} - \sum_{i = \ell }^{|\I|} (p_i - p_{\ell -1}) z^\ast_i
	= \left(p_{\ell}  - p_{\ell -1}\right) \cdot \left( 1 - \sum_{i = \ell }^{|\I|} z^\ast_{i}\right).
\end{equation*}
Therefore,
\begin{itemize}
	\item[(i)] if 
	\begin{equation}\label{sepineq1}
		\sum_{i =1}^{|\I|} z^\ast_i < 1,
	\end{equation}	
	then we have $F_{0}(z^\ast) \le F_{1}(z^\ast) \le \cdots \le F_{|\I|}(z^\ast)$, 
	and thus $\ell^* = 0$;
	\item[(ii)] otherwise, we have
	$F_{0}(z^\ast) \ge \cdots \ge F_{\ell^*}(z^\ast)$ and 
	$F_{\ell^*}(z^\ast) \le \cdots \le F_{|\I|}(z^\ast)$, where 
	$\ell^* \in \{1, \dots, |\I|\}$ satisfies 
	\begin{equation}\label{sepineq2}
		\sum_{i = \ell^* +1}^{|\I|} z^\ast_{i} < 1 \le \sum_{i = \ell^*}^{|\I|} z^\ast_{i}.
	\end{equation}
\end{itemize}
Since checking whether \eqref{sepineq1} holds or determining $\ell^*$ with \eqref{sepineq2} can all be accomplished with a complexity of $\CO(|\I|)$, the separation of inequalities \eqref{cons:zeta-linear} can  be performed with a  complexity of $\mathcal{O}(|\I|)$.

%% file: EOA_example.tex
\section{Strength of the enhanced outer-approximation inequalities \eqref{ineq:coef-str-step2} over the classic ones \eqref{cons:prod-outer-approx-explicit}}
\label{sect:appendix-EOA-example}

\begin{example}
	Let $\I = \{1,2,3\}$,
	$(p_1, p_2, p_3) = (1 - 1/e, 1 - 1/e^3, 1)$,
	and $y^\ast = (1, 0, 0)$ (where $e$ denotes the base of the natural
	logarithm). From the definitions of $a_i(y^\ast)$ and $c(y^\ast)$ in \eqref{def:c}--\eqref{def:ai}, we have 
	\begin{equation*}
		\begin{aligned}
			& a_1(y^\ast) 
			= - [\ln(1-p_1)] (1-p_1)^{y^\ast_1} (1-p_2)^{y^\ast_2} 
			= - [\ln(1-p_1)] (1-p_1) 
			= 1/e, \\
			& a_2(y^\ast) 
			= - [\ln(1-p_2)] (1-p_1)^{y^\ast_1} (1-p_2)^{y^\ast_2} 
			= - [\ln(1-p_2)] (1-p_1)
			= 3 / e,\\
			& c(y^\ast)
			= 1 - (1-p_1)^{y^\ast_1} (1-p_2)^{y^\ast_2} 
			+ (1-p_1)^{y^\ast_1} (1-p_2)^{y^\ast_2} \left([\ln(1-p_1)] y_1^\ast + [\ln(1-p_2)] y_2^\ast\right)\\
			& \phantom{c(y^\ast)} = 1 - (1-p_1) + (1-p_1) [\ln(1-p_1)] = 1 - 2 / e.
		\end{aligned}
	\end{equation*}
	Thus, inequality \eqref{cons:prod-outer-approx-explicit} defined by $y^\ast$ reduces to
	\begin{equation}\label{oaex}
		\eta 
		\le c(y^\ast) + a_1(y^\ast) y_1 + a_2(y^\ast) y_2  +  y_3 
		= 1 - \frac{2}{e} + \frac{1}{e} \ y_1 + \frac{3}{e} \  y_2 +  y_3.
	\end{equation}
	By $a_1(y^\ast) < 1-c(y^\ast)$, $a_2(y^\ast) > 1-c(y^\ast)$, 
	we obtain $\CL = \{ 2 \}$, and thus
	inequality \eqref{ineq:coef-str-step2} reduces to 
	\begin{equation*}
		\eta \le c(y^\ast) + a_1(y^\ast) y_1 + (1-c(y^\ast))(z_2 + z_3)
		= 1 - \frac{2}{e} + \frac{1}{e} \ y_1 + \frac{2}{e} \ (z_2 + z_3),
	\end{equation*}
	which is stronger than inequality \eqref{oaex}.
\end{example}

%% file: auxilary_subadditive.tex
\section{Proof of \cref{lem:foundations}}
\label{sect:appendix-auxilary-subadditive}
\begin{myproof}
	We first prove statement (i).
	If $z_i = 0$, then $y_i = 0$ and thus $1 - (1 - p_i)^{y_i} = 0 =  h_{i, k}(y_i, z_i)$.
	Otherwise, $z_i = 1$ and $1 \leq y_i \leq K$.
	In this case, if $y_i = k$, by \eqref{def:h_k}, $h_{i, k}(y_i, z_i) =1 - (1 - p_i)^{k} = 1 - (1 - p_i)^{y_i} $. 
	Otherwise,  we can rewrite $1 - (1 - p_i)^{y_i} \le h_{i, k}(y_i, z_i)$ as 
	\begin{align}
		\frac{1 - (1-p_i)^{y_i} - (1 - (1-p_i)^{k})}{y_i - k}\le \frac{1 - (1-p_i)^{k+1} - (1 - (1-p_i)^{k})}{k+1 - k}, \quad {\text{if}}~y_i \geq k+1, \label{tmpeq1}\\
		\frac{1 - (1-p_i)^{k} - (1 - (1-p_i)^{y_i})}{k - y_i} \ge \frac{1 - (1-p_i)^{k+1} - (1 - (1-p_i)^{k})}{k+1 - k} ,\quad {\text{if}}~y_i \leq k - 1.\label{tmpeq2}
	\end{align}
	Both \eqref{tmpeq1} and \eqref{tmpeq2} follow from the concavity of function $f(x) = 1 - (1 - p_i)^x$, that is, 
	$\frac{f(a) - f(b)}{a - b} \ge \frac{f(c) - f(d)}{c - d}$ for all $a$, $b$, $c$, $d \in \R_{\rev{\ge 0}}$ with $a > b$, $c > d$, $a \le c$, and $b \le d$.

	Next, we show that statement (ii) holds.
	If $z_i = y_i$, then $1 - (1 - p_i)^{y_i - z_i} = 0 =  p_i (y_i - z_i)$. 
	Otherwise,  $y_i - z_i \ge 1$ must hold (as $(y_i, z_i) \in \Z_{\rev{\ge 0}} \times \{0,1\}$  and $y_i \geq z_i$).
	Combining with the concavity of function $f$, we obtain 
	\begin{equation}\label{tmpeq3}
		\frac{1 - (1 - p_i)^{y_i - z_i} - 0}{y_i - z_i - 0} \le \frac{1 - (1 - p_i)^{1} - 0}{1 - 0},
	\end{equation}
	or equivalently, $1 - (1 - p_i)^{y_i - z_i} \le p_i (y_i - z_i)$, which completes the proof of statement (ii).
\end{myproof}

%% file: LSI_facet_defining.tex
\section{Proof of \cref{thm:strong-valid-X}}
\label{sect:appendix-LSI-facet-defining}

\begin{myproof}
	Let $\boldsymbol{0}$ and $\boldsymbol{e}^i$ denote the all-zero and $i$-th unit vectors of dimension $|\I|$.
	Observe that $(\eta,y,z)=(0, \boldsymbol{0}, \boldsymbol{0})$, $(-1, \boldsymbol{0}, \boldsymbol{0})$,
	$\{(p_i, \boldsymbol{e}^i, \boldsymbol{e}^i)\}_{i \in \I}$, and
	$\{(1-(1-p_i)^2, 2\boldsymbol{e}^i, \boldsymbol{e}^i)\}_{i \in \I}$ 
	are $2|\I|+2$ affinely independent points in $\X$, implying that $\conv(\X)$ is full-dimensional.
	Thus, to prove the statement, it suffices to show that there exist $2n+1$ affinely independent points in $\X$ satisfying \eqref{ineq:goal-lift-and-proj-C-lifted} at equality.
	Without loss of generality, we assume that $\C= \{1, \ldots, |\C|\}$, $\I_{\Partial}\backslash \C = \{	 |\C| + 1, \dots,  |\I_{\Partial}|\}$, and $\I_{\Full} = \{|\I_{\Partial}| + 1, \dots, |\I| \}$. 
	Let $i_1 \in \I_{\Full}$ and consider the $2|\I|+1$ points listed in \cref{table:affinely-independent-points}.
	\begin{table}[htbp]
		\caption{The {$2n+1$} affinely independent points in $\X$ satisfying \eqref{ineq:goal-lift-and-proj-C-lifted} at equality.}
		\label{table:affinely-independent-points}
		\centering
		\renewcommand{\arraystretch}{2}
		\begin{tabular}{|l|l|}
			\hline
			Points & Range  \\
			\hline
			$\left(1 - \rev{p_{\C}}, \sum_{i \in \C} \be^i, \sum_{i \in \C} \be^i \right)$
			& --
			\\ \hline
			$\left(1 - (1 - p_{\ell}) \rev{p_{\C}}, \sum_{i \in \C} \be^i +  \be^{\ell}, \sum_{i \in \C} \be^i \right)$
			& $\forall~\ell = 1, \dots, |\C|$ 
			\\ \hline 
			$\left(1, \be^{i_1} + \be^{\ell}, \be^{i_1} + \be^{\ell}\right)$
			& $\forall~\ell = 1, \dots, |\C|$ 
			\\ \hline
			$\left(1 - (1 - p_{\ell})^{k_\ell}\rev{p_{\C}}, \sum_{i \in \C} \be^i + k_\ell \be^{\ell}, \sum_{i \in \C} \be^i + \be^{\ell} \right)$
			& $\forall~\ell = |\C| + 1, \dots, |\I_{\Partial}|$
			\\ \hline
			$\left(1 - (1 - p_{\ell})^{k_\ell+1}\rev{p_{\C}}, \sum_{i \in \C} \be^i + (k_\ell+1) \be^{\ell}, \sum_{i \in \C} \be^i + \be^{\ell} \right)$
			& $\forall~\ell = |\C| + 1, \dots, |\I_{\Partial}|$
			\\ \hline
			$\left(1,  \be^{\ell}, \be^{\ell}\right)$
			& $\forall~\ell = |\I_{\Partial}| + 1, \dots,|\I|$
			\\ \hline
			$\left(1,  2\be^{\ell}, \be^{\ell}\right)$
			& $\forall~\ell = |\I_{\Partial}| + 1, \dots,|\I|$
			\\ \hline
		\end{tabular}
	\end{table}
	By simple computations, points defined in \cref{table:affinely-independent-points} are affinely independent points in $\X$ satisfying \eqref{ineq:goal-lift-and-proj-C-lifted} at equality, which {proves} the statement.
\end{myproof}

%% file: sepa_nphardness.tex
\section{Proof of \cref{thm:np-hard}}
\label{sect:appendix-nphard}


\begin{myproof}
We shall prove the NP-hardness of problem \eqref{prob:sepa-C}, or its equivalent maximization version,
\begin{equation}\label{SEP}
	\max_{\C \subseteq \I_{\Partial}} ~ 
	p_\C \cdot 
	\left(
	1 - 
	\sum_{i \in \I_{\Partial} \backslash \C} h_{i, k_i^*}(y_i^*, z_i^*)
	- 
	\sum_{i \in \C} p_i (y_i^* - z_i^*)
	-  
	\sum_{i \in \I_{\Full}} z_i^*
	\right),
\end{equation}
by establishing a polynomial time reduction from the partition problem, which is NP-complete
\ifarxiv
(Garey and Johnson, 1979).
\else
\citep{Garey1979}.
\fi
First, we introduce the partition problem: 
given a finite set $[n] = \{1, 2, \dots, n\}$ and a value $a_i \in \mathbb{Q}_{\rev{\le 0}}$ for the $i$-th element with $\sum_{i \in [n]} a_i = -2$, 
does there exist a subset $\C \subseteq [n]$ such that $\sum_{i \in \C} a_i = -1$?
Without loss of generality, we assume $a_i \in (-1, 0)$ for all $i \in [n]$.

Given any instance of the partition problem, 
we construct an instance of problem \eqref{SEP}
by setting $\I_{\Full} = \varnothing$, $\I_{\Partial} = [n]$,  
$p_i = 1 - e^{a_i} \in (0,1)$, and
$y_i^* = z_i^* = - \frac{a_i}{2 (1 - e^{a_i})}  \in (0,1)$ for $i \in [n]$\footnote{Observe that $(\min_{x \in [-1, 0]} 2 (1 - e^x) + x) = 0$, where the only minimum is achieved at $x=0$. Thus, {$ - \frac{a_i}{2 (1 - e^{a_i})}<1$} holds for  $a_i \in (-1, 0)$.
}.
By definition, $k_i^*$ defined in \eqref{def:min-hk-index} reduces to $k_i^* = \lfloor y_i^* / z_i^* \rfloor = 1$ where $i \in [n]$.
Thus, it follows from  \eqref{def:h_k} that
$$h_{i, k_i^*}(y_i^*, z_i^*) = p_i z_i^* = - \frac{a_i}{2}, \ \forall \ i \in \I_{\Partial} \backslash \C.$$
Therefore, \eqref{SEP} reduces to 
\begin{equation}\label{SEP2}
	\max_{\C \subseteq \I_{\Partial}} ~ 
	\prod_{i \in \C} e^{a_i} \cdot \left(
	1 + \frac{1}{2} \sum_{i \in \I_{\Partial} \backslash \C} a_i
	\right)
	\overset{(a)}{=} 	\max_{\C \subseteq \I_{\Partial}} ~ \prod_{i \in \C} e^{a_i} \cdot \left(
	-  \frac{1}{2} \sum_{i \in \C} a_i
	\right).
\end{equation}
where (a) follows from $\sum_{i \in [n]} a_i = \sum_{i \in \I_{\Partial} \backslash \C} a_i + \sum_{i \in \C} a_i  = -2$.
To solve problem \eqref{SEP2}, we consider the following two cases. 
\begin{itemize}
	\item [(i)] If $\sum_{i \in \C} a_i = 0$, then the objective value of \eqref{SEP2} is equal to $0$.
	\item [(ii)] Otherwise, by applying the logarithm on the objective function, we obtain 
	an equivalent problem of \eqref{SEP2}: 
	\begin{equation}\label{SEP3}
		\max_{\C \subseteq \I_{\Partial}, \, \C \neq \varnothing} ~ 
		\sum_{i \in \C} a_i + \ln \left(
		-  \frac{1}{2} \sum_{i \in \C} a_i
		\right).
	\end{equation}
	It is simple to check that 
	$$\left(\max_{x \in [-2,0)}  x + \ln\left(-\frac{x}{2}\right)\right) = -1 + \ln \frac{1}{2},$$
	where the only maximum point is arrived at $x = -1$.
	This, together with $\sum_{i \in \C} a_i \in [-2,0)$, implies that the optimal value of problem \eqref{SEP3} is equal to $-1+ \ln \frac{1}{2}$ if and only if $\sum_{i \in \C} a_i =-1$.
\end{itemize}
\noindent
Combining the above two cases, we can conclude that the optimal value of \eqref{SEP2} is   $e^{-1 + \ln \frac{1}{2}} > 0$ 
if and only if $\sum_{i \in \C} a_i = -1$, or equivalently, the answer to the partition problem is yes.
Since the above transformation can be done in polynomial time and the partition problem is NP-complete, we conclude that problem \eqref{SEP} is NP-hard.
\end{myproof}

%% file: proof_heur_sepa.tex
\section{Proof of \cref{thm:C-deterministic}}
\label{sect:appendix-sepa-C}
\begin{myproof} Observe that the objective function {$\nu(\C)$} of problem \eqref{prob:sepa-C} (i.e., the right-hand side of \eqref{ineq:goal-lift-and-proj-C-lifted}) at $(\eta^*, y^*, z^*)$ can be written as $\nu(\C) = 1 - \rev{p}_{\C}  \rev{\omega}_{\C}$, where 
\begin{align}\label{def:B_C}
	& \rev{\omega}_{\C} = 1  - \left(
	\sum_{i \in \I_{\Partial} \backslash \C} h_{i, k_i^*}(y_i^*, z_i^*)
	+ 
	\sum_{i \in \C} p_i (y_i^* - z_i^*)
	+ 
	\sum_{i \in \I_{\Full}} z_i^*
	\right).
\end{align}
	We first prove statement (i).
	By $y_{i_0}^* = z_{i_0}^* = 1$, it follows $p_{i_0} (y_{i_0}^* - z_{i_0}^*) = 0$ and $k_{i_0}^* = 1$.
	As a result,
	$h_{i_0, k_{i_0}^*}(y_{i_0}^*, z_{i_0}^*) = h_{i_0, 1}(y_{i_0}^*, z_{i_0}^*) = p_{i_0}(1-p_{i_0}) y_{i_0}^* + (1  -(1-p_{i_0})(p_{i_0}+1)) z_{i_0}^* = p_{i_0}$,  
	$\rev{\omega}_{\C \cup \{i_0\}} = \rev{\omega}_{\C} + h_{i_0, k_{i_0}^*}(y_{i_0}^*, z_{i_0}^*) - p_{i_0} (y_{i_0}^* - z_{i_0}^*) = \rev{\omega}_{\C} + p_{i_0}$, and 
	\begin{equation}\label{tmp}
		\rev{\omega}_{\C} \overset{(a)}{\leq}  1 - h_{i_0, k_{i_0}^*}(y_{i_0}^*, z_{i_0}^*) = 1-p_{i_0}, 
	\end{equation}
	where (a) follows from $0 \leq z_i^* \leq y_i^*\leq Kz_i^*$ for $i \in \I$ and  $h_{i, k_i^*}(y_i^*, z_i^*)\geq 0$ for $i \in \I_{\Partial}\backslash (\C\cup\{i_0\})$ (by \cref{lem:min-hk}). 
	Hence,  
	\begin{equation*}
		\begin{aligned}
			\nu(\C \cup \{i_0\})-  \nu(\C)
			& {=}  
			- \rev{p}_{\C \cup \{i_0\}} \rev{\omega}_{\C \cup \{i_0\}} + \rev{p}_{\C}  \rev{\omega}_{\C}\\
			& \overset{(a)}{=} - \rev{p}_{\C}  (1 - p_{i_0})  \left(\rev{\omega}_{\C} + p_{i_0} \right) + \rev{p}_{\C}  \rev{\omega}_{\C}\\
			& = - \rev{p}_{\C}  p_{i_0}  (1 - p_{i_0} - \rev{\omega}_{\C}) \overset{\text{(b)}}{\le} 0, 
		\end{aligned}
	\end{equation*}
	where (a) follows from $p_{\C\cup \{i_0\}}= \prod_{i \in \C\cup \{i_0\}} (1-p_i)= (1-p_{i_0})\prod_{i \in \C} (1-p_i)=(1-p_{i_0})p_{\C}$,
	(b) follows from $\rev{p}_{\C} > 0$, {$p_{i_0} \geq 0$} (as $i_0 \in \I_{\Partial}$), and \eqref{tmp}.
	This shows statement (i).
	
	Next, we prove statement (ii).
	As inequality \eqref{ineq:goal-lift-and-proj-C-lifted} defined by $\C$ and $\{k_i^*\}_{i \in \I_{\Partial}\backslash \C}$ is violated by $(\eta^*, y^*, z^*)$, it must follow $\rev{\nu}(\C) = 1 - \rev{p}_{\C}  \rev{\omega}_{\C} < \eta^* \leq 1$. 
	This, together with $\rev{p}_\C > 0$, implies $\rev{\omega}_{\C}> 0$.
	By $y_{i_0}^* = z_{i_0}^* = 0$, we have $p_{i_0} (y_{i_0}^* - z_{i_0}^*) = 0$, $k_{i_0}^* = 1$, and
	$h_{i_0, k_{i_0}^*}(y_{i_0}^*, z_{i_0}^*) = h_{i_0, 1}(y_{i_0}^*, z_{i_0}^*) = p_{i_0}(1-p_{i_0}) y_{i_0}^* + (1  -(1-p_{i_0})(p_{i_0}+1)) z_{i_0}^* = 0$.
	Thus, 
	$\rev{\omega}_{\C} = \rev{\omega}_{\C\backslash \{i_0\}} + h_{i_0, k_{i_0}^*}(y_{i_0}^*, z_{i_0}^*) - p_{i_0} (y_{i_0}^* - z_{i_0}^*) = \rev{\omega}_{\C \backslash \{i_0\}}$ and
	\begin{align*}
		\nu(\C)-  \nu(\C \backslash \{i_0\})
		& =  
		- \rev{p}_{\C}  \rev{\omega}_{\C} + \rev{p}_{\C \backslash \{i_0\}}  \rev{\omega}_{\C \backslash \{i_0\}}\\
		& \overset{(a)}{=} 
		- \rev{p}_{\C \backslash \{i_0\}}  (1 - p_{i_0})   \rev{\omega}_{\C} +  \rev{p}_{\C \backslash \{i_0\}}  \rev{\omega}_{\C} \\
		& = p_{i_0}  \rev{p}_{\C \backslash \{i_0\}}   \rev{\omega}_{\C} \overset{\text{(a)}}{\ge} 0, 
	\end{align*}
	where (a) follows from $p_{\C}= \prod_{i \in \C} (1-p_i)= (1-p_{i_0})\prod_{i \in \C\backslash \{i_0\}} (1-p_i)=(1-p_{i_0})p_{\C\backslash \{i_0\}}$, 
    and (b) holds as $\rev{p}_{\C\backslash \{i_0\}} > 0$, $p_{i_0} \geq 0$ (as $i_0 \in \I_{\Partial}$), and $ \rev{\omega}_{\C} >0$.
	Therefore, statement (ii) holds.
\end{myproof}

%% file: alg_heur_sepa.tex
\section{Heuristic separation procedure for the lifted subadditive inequalities \eqref{ineq:goal-lift-and-proj-C-lifted}}
\label{sect:appendix-alg-heur-sepa}

\begin{algorithm}[H]
	\caption{A local search algorithm for separating the lifted subadditive inequalities}
	\label{alg:sepa2}
	\KwIn{Point $(\eta^*, y^*, z^*) \in \R \times \R_{\rev{\ge 0}}^{|\I|} \times [0,1]^{|\I|}$ with $\eta^* \leq 1$ and $z_i^* \le y_i^* \le K z_i^*$ for $i \in \I$.}
	\For{$i \in \I_{\Partial}$}{Set $k_i^* \leftarrow \lfloor y_i^* / z_i^*\rfloor$\;}
	Set $\C \leftarrow \I_{\Partial}^1$ and compute $\nu(\C)$\; 
	\Repeat{no better solution is found}
	{Compute $\C^* \in \argmin \{ \rev{\nu}(\C')\, : \, \C' \in \mathfrak{N}(\C)\}$\;
		\If{$ \nu(\C^*) <\nu(\C)$}{
			Set $\C \leftarrow \C^*$ and mark that a better solution is found\; 
		}
	}
	\eIf{$ \nu(\C) <\eta^*$}
	{
		{\bf Output} the violated lifted subadditive inequality \eqref{ineq:goal-lift-and-proj-C-lifted} defined by $\C$ and $\{k_i^*\}_{i \in \I_{\Partial} \backslash \C}$\;
	}
	{
		\rev{The local search algorithm finds no violated lifted subadditive inequality \eqref{ineq:goal-lift-and-proj-C-lifted}\;}
	}
\end{algorithm}

%% file: newly_solved_instances.tex
\section{Computational results of the newly solved instances}
\label{sect:newly-solved-instances}

\begin{table}[htbp]
	\footnotesize
	\caption{Previously unsolved \MPCLP instances in \citet{Alvarez-Miranda2019} solved to optimality by {the proposed \BnC algorithm}.}
	\centering
	\label{table:unsolved}
	\input{table_unsolved}
\end{table}

%% file: table_unsolved.tex
\begin{tabular}{cccrrr@{\hspace{20pt}}cccrrr}
	\toprule
	\multicolumn{1}{r}{\idfull}
	& \multicolumn{1}{r}{\V}
	& \multicolumn{1}{r}{\K}
	&  \tbltime &  \tblnode &  \tblobj & \multicolumn{1}{r}{\idfull}
	& \multicolumn{1}{r}{\V}
	& \multicolumn{1}{r}{\K}
	&  \tbltime &  \tblnode &  \tblobj \\
	\midrule
	15-5-20-0.2 &     300 &     100 &    58.9 &    7913 &  250.45    & 15-10-25-0.2 &     300 &     100 &   875.1 &   54901 &  286.01 \\
	18-5-20-0.2 &     400 &      40 &   260.6 &    7545 &  238.05    & 18-10-25-0.2 &     400 &      40 &   630.4 &    9988 &  308.98 \\
	19-5-20-0.2 &     400 &      80 &  1310.4 &   46250 &  323.16    & 19-10-25-0.2 &     400 &      80 &   575.6 &   15017 &  375.66 \\
	27-5-20-0.2 &     600 &      10 &   214.8 &    1977 &  305.54    & 20-10-25-0.2 &     400 &     133 &   875.2 &   41068 &  397.16 \\
	32-5-20-0.2 &     700 &      10 &    37.4 &     215 &  374.85    & 25-10-25-0.2 &     500 &     167 &    31.3 &    2496 &  499.70 \\
	35-5-20-0.2 &     800 &       5 &    20.2 &      65 &  460.70    & 26-10-25-0.2 &     600 &       5 &    30.7 &     192 &  370.60 \\
	38-5-20-0.2 &     900 &       5 &   216.4 &     393 &  563.03    & 27-10-25-0.2 &     600 &      10 &   834.0 &   20281 &  446.98 \\
	39-5-20-0.2 &     900 &      10 &  1186.3 &    8692 &  682.84    & 31-10-25-0.2 &     700 &       5 &   360.7 &     911 &  510.57 \\
	19-5-20-0.5 &     400 &      80 &    38.1 &    1948 &  317.62    & 32-10-25-0.2 &     700 &      10 &   741.7 &    2215 &  553.67 \\
	20-5-20-0.5 &     400 &     133 &    78.0 &    5131 &  362.51    & 34-10-25-0.2 &     700 &     140 &     7.5 &      93 &  699.83 \\
	23-5-20-0.5 &     500 &      50 &   212.1 &    2263 &  354.87    & 35-10-25-0.2 &     800 &       5 &    51.4 &     159 &  668.82 \\
	24-5-20-0.5 &     500 &     100 &   216.1 &    6629 &  430.74    & 38-10-25-0.2 &     900 &       5 &   844.5 &    2245 &  788.21 \\
	25-5-20-0.5 &     500 &     167 &   360.3 &   14824 &  476.08    & 23-10-25-0.5 &     500 &      50 &  2277.6 &   33626 &  436.67 \\
	27-5-20-0.5 &     600 &      10 &    86.4 &     793 &  290.94    & 24-10-25-0.5 &     500 &     100 &   413.9 &   18185 &  486.77 \\
	32-5-20-0.5 &     700 &      10 &    16.0 &      88 &  357.83    & 27-10-25-0.5 &     600 &      10 &   146.2 &    1133 &  431.38 \\
	35-5-20-0.5 &     800 &       5 &    17.1 &      45 &  432.06    & 29-10-25-0.5 &     600 &     120 &   499.8 &   27947 &  594.67 \\
	36-5-20-0.5 &     800 &      10 &   998.4 &    1959 &  453.18    & 31-10-25-0.5 &     700 &       5 &    91.4 &     274 &  490.91 \\
	38-5-20-0.5 &     900 &       5 &   113.7 &     247 &  526.27    & 32-10-25-0.5 &     700 &      10 &    67.3 &     303 &  533.12 \\
	39-5-20-0.5 &     900 &      10 &   105.3 &     282 &  637.70    & 34-10-25-0.5 &     700 &     140 &     5.8 &      29 &  699.59 \\
	28-5-20-0.8 &     600 &      60 &   207.9 &    1696 &  472.58    & 35-10-25-0.5 &     800 &       5 &    41.0 &      91 &  640.36 \\
	29-5-20-0.8 &     600 &     120 &    32.0 &     733 &  537.05    & 38-10-25-0.5 &     900 &       5 &   256.9 &     534 &  755.86 \\
	30-5-20-0.8 &     600 &     200 &    25.8 &    1370 &  577.40    & 28-10-25-0.8 &     600 &      60 &   539.6 &   32517 &  559.02 \\
	33-5-20-0.8 &     700 &      70 &   184.5 &    2655 &  581.92    & 33-10-25-0.8 &     700 &      70 &   962.0 &   48173 &  673.22 \\
	34-5-20-0.8 &     700 &     140 &     9.8 &      75 &  653.35    & 34-10-25-0.8 &     700 &     140 &     5.7 &      39 &  699.36 \\
	36-5-20-0.8 &     800 &      10 &   339.2 &     692 &  427.42    & 36-10-25-0.8 &     800 &      10 &  3462.4 &   18927 &  621.44 \\
	37-5-20-0.8 &     800 &      80 &  1564.3 &   71184 &  683.42    & 37-10-25-0.8 &     800 &      80 &   155.2 &    1495 &  777.42 \\
	38-5-20-0.8 &     900 &       5 &    82.8 &     106 &  490.74    & 38-10-25-0.8 &     900 &       5 &   105.7 &     189 &  723.50 \\
	39-5-20-0.8 &     900 &      10 &   119.5 &     214 &  592.56    & 40-10-25-0.8 &     900 &      90 &  1586.8 &   80429 &  889.92 \\
	40-5-20-0.8 &     900 &      90 &   125.8 &     859 &  801.05    &  &  &  &  &  &  \\
	\bottomrule
\end{tabular}